\documentclass[11pt]{amsart}

\usepackage{amsmath,amsfonts,amsthm,amssymb,graphicx,enumerate,tikz}
\usepackage[all,2cell,ps]{xy}
\usepackage[title]{appendix}

\title{Amalgam Anosov representations}

\author[R. D. Canary]{Richard D. Canary}
\address{Department of Mathematics, University of Michigan, Ann Arbor, MI, USA}
\email{canary@umich.edu}

\author[M. Lee]{Michelle Lee}
\address{Department of Mathematics, University of Maryland, College Park, MD, USA}
\email{mdl@umd.edu}

\author[M. Stover]{Matthew Stover}
\address{Department of Mathematics, Temple University, Philadelphia, PA, USA}
\email{mstover@temple.edu}

\author[Appendix with A. Sambarino]{{\rm Appendix by the authors and} Andr\'es Sambarino}
\address{Universit\'e Paris IV Pierre et Marie Curie, Paris, France}
\email{andres.sambarino@gmail.com}

\theoremstyle{plain}
\newtheorem{Theorem}{Theorem}[section]
\newtheorem{Lemma}[Theorem]{Lemma}
\newtheorem{Proposition}[Theorem]{Proposition}
\newtheorem{Corollary}[Theorem]{Corollary}
\newtheorem*{Conjecture}{Conjecture}
\newtheorem{Definition}[Theorem]{Definition}

\theoremstyle{definition}
\newtheorem*{Remark}{Remark}

\DeclareMathOperator{\PSL}{PSL}
\DeclareMathOperator{\Out}{Out}
\DeclareMathOperator{\Mod}{Mod}

\DeclareMathOperator{\SL}{SL}

\newcommand{\R}{\mathbb{R}}
\newcommand{\C}{\mathbb{C}}
\newcommand{\Z}{\mathbb{Z}}

\newcommand{\X}{\mathbb{X}(\Gamma, G)}
\newcommand{\Xa}{\mathbb{X}_{\mathrm{A}}(\Gamma, G)}
\newcommand{\Xaa}{\mathbb{X}_{\mathrm{A^2}}(\Gamma, G)}
\newcommand{\Xsa}{\mathbb{X}_{\mathrm{SA}}(\Gamma, G)}
\newcommand{\RPd}{\mathbb{P}(\R^d)}
\newcommand{\RPddual}{\mathbb{P}(\R^d)^*}

\begin{document}

\begin{abstract}
Let $\Gamma$ be a one-ended, torsion-free hyperbolic group and let $G$ be a semisimple Lie group with finite center. 
Using the canonical JSJ splitting due to Sela, we define amalgam Anosov representations of $\Gamma$ into $G$ and prove that they form a domain of discontinuity for the action of $\Out(\Gamma)$. 
In the appendix, we prove, using projective Anosov Schottky groups, that if the restriction of the representation to every Fuchsian or rigid vertex group of the JSJ splitting of $\Gamma$ is Anosov, with respect to a fixed pair of opposite
parabolic subgroups, then $\rho$ is amalgam Anosov. \end{abstract}

\maketitle

\section{Introduction}\label{sec:intro}

We study the dynamics of the action of the outer automorphism group $\Out(\Gamma)$ of a torsion-free word hyperbolic group $\Gamma$ on the space 
\[
\X={\rm Hom}(\Gamma,G)/G
\]
of (conjugacy classes of) representations of $\Gamma$ into a semisimple Lie group $G$. The subject is motivated by Fricke's result that the mapping class group $\Mod(S)$ of a closed, oriented hyperbolic surface $S$ acts properly discontinuously on the Teichm\"uller space ${\mathcal{T}}(S)$ of marked hyperbolic structures on $S$, which can be identified with a connected component of $\mathbb{X}(\pi_1(S),\PSL_2(\R))$.

Labourie \cite{Labourie} introduced the notion of an Anosov representation of a word hyperbolic group $\Gamma$ into
a semisimple Lie group $G$ with respect to a pair $P^\pm$ of opposite parabolic subgroups of $G$. We will call a representation Anosov if it is
Anosov with respect to some pair of opposite parabolic subgroups.
If $G$ has rank one, then a representation is Anosov if and only if it is convex cocompact, so Anosov representations are natural generalizations of convex cocompact representations into the higher rank setting (see Guichard--Wienhard \cite[Prop.\ 5.15]{Guichard-Wienhard}). Labourie \cite[Thm.\ 1.0.2]{Labourie2} and Guichard--Wienhard \cite[Cor. 5.4]{Guichard-Wienhard} (see also Canary \cite[Thm.\ 6.2]{canary-survey}) proved that $\Out(\Gamma)$ acts properly discontinuously on the open subset $\Xa\subset\X$ of Anosov representations. One might naturally ask whether $\Xa$ is a maximal domain of discontinuity for the action of $\Out(\Gamma)$ on $\X$.

In this paper we define amalgam Anosov representations of a one-ended, torsion-free hyperbolic group $\Gamma$ into a semisimple Lie group $G$ and
prove that the space $\Xaa$ of amalgam Anosov representations is a domain of discontinuity for the action of $\Out(\Gamma)$ on $\X$ that contains $\Xa$.
In Section \ref{sec:ex}, we exhibit examples where $\Xaa$ is strictly larger than $\Xa$.

\begin{Theorem}\label{thm:main}
Let $\Gamma$ be a one-ended, torsion-free hyperbolic group and let $G$ be a semisimple Lie group with finite center. Then the subset $\Xaa$ of $\X$ consisting of amalgam Anosov representations is an open $\Out(\Gamma)$-invariant subset of $\X$ such that
\begin{enumerate}

\item
$\Xaa$ contains the space $\Xa$ of Anosov representations, and

\item
the action of $\Out(\Gamma)$ on $\Xaa$ is properly discontinuous.

\end{enumerate}
\end{Theorem}

Our definition of amalgam Anosov representations is based on Sela's canonical JSJ splitting of a one-ended, torsion-free hyperbolic group $\Gamma$, see Sela \cite{Sela} and Bowditch \cite{Bowditch}.
Roughly, a JSJ splitting is a graph of groups decomposition for $\Gamma$ such that each edge group is infinite cyclic and each vertex group is either maximal cyclic, Fuchsian, or rigid (i.e., admits no further $\Z$-splitting consistent with the given splitting). 
The JSJ splitting is preserved by every automorphism of $\Gamma$ (see Section \ref{sec:jsj} for a precise statement), and therefore may be used to analyze $\Out(\Gamma)$. In particular,
there is a finite index subgroup of $\Out(\Gamma)$ that is a central extension of the product of the mapping class groups of the Fuchsian vertex groups by
a free abelian group generated by twists in the cyclic vertex groups (see Sela \cite{Sela} and Levitt \cite{levitt}).

A representation $\rho:\Gamma\to G$ is \emph{amalgam Anosov} if (1) the restriction of $\rho$ to each Fuchsian vertex group is Anosov, and (2) for every cyclic 
vertex group $\Gamma_v$ there exists a free subgroup $H$ of $\Gamma$ ``registering'' $\Gamma_v$ such that $\rho|_H$ is Anosov.
Here $H$ is registering in the sense that, up to finite index, the group of twists in the cyclic vertex
group $\Gamma_v$ preserves $H$ (up to conjugacy) and acts effectively on $H$. See Section \ref{aareps} for a precise definition.

The key idea in the proof is that any sequence of distinct elements in $\Out(\Gamma)$ has a subsequence that ``projects to'' a sequence of
distinct elements in either the mapping class group of some Fuchsian vertex group $\Gamma_v$ or in the group of twists in some cyclic vertex group $\Gamma_v$.
In the first case, the fact that $\Out(\Gamma_v)$ acts properly discontinuously on ${\mathbb X}_{\mathrm{A}}(\Gamma_v,G)$ guarantees that any orbit of the
subsequence exits every compact subset of $\Xaa$, while in the second case we use the fact that $\Out(H)$ acts properly discontinuously on ${\mathbb X}_{\mathrm{A}}(H,G)$,
where $H$ is a subgroup registering $\Gamma_v$.

The condition that the restriction of $\rho$ to various subgroups is Anosov is used in two distinct ways. First, the fact that Anosov representations are
quasi-isometric embeddings (or, equivalently in this setting, are well-displacing) is the key ingredient in the proof of proper discontinuity.
Second, the fact that the set of Anosov representations is open, guarantees that $\Xaa$ is open. Neither fact relies on the choice of the pair of opposite
parabolic subgroups, so we may allow different choices of parabolic subgroups for different subgroups in the definition.

\medskip

One simple way to construct amalgam Anosov representations that are not Anosov is to glue together Anosov representations of vertex subgroups. We say that $\rho:\Gamma\to G$ is \emph{strongly amalgam Anosov} if there exists a pair $P^\pm$ of opposite parabolic subgroups of $G$ such that the restriction of $\rho$ to every rigid or Fuchsian vertex subgroup is \hbox{$(P^+,P^-)$-Anosov}. In the appendix, we prove:

\medskip\noindent
{\bf Theorem \ref{main-appendix}.}
{\em
Suppose that $\Gamma$ is a one-ended, torsion-free hyperbolic group and $G$ is a semisimple Lie group with finite center. If $\rho \in \X$ is strongly amalgam Anosov, then $\rho$ is amalgam Anosov.
}
\medskip

The key tool in the proof of Theorem \ref{main-appendix} is that, given a generic collection of finitely many biproximal elements of $\SL_d(\R)$, a group generated by sufficiently high powers of these elements is a projective Anosov Schottky group. This is a natural generalization of the classical fact that given any two hyperbolic elements of $\PSL_2(\C)$ with disjoint fixed point sets, the group generated by sufficiently high powers of these elements is a Schottky group. See Theorem \ref{thm:ConvexSchottky} for a precise statement. We prove Theorem \ref{thm:ConvexSchottky} using work of Benoist \cite{BenoistAnnals,Benoist}, Quint \cite{QuintDynam} and Gu\'eritaud--Guichard--Kassel--Wienhard \cite{GGKW}. 
Kapovich, Leeb and Porti \cite[Theorem 7.40]{KLP} earlier gave a proof of  a more general version of 
Theorem \ref{thm:ConvexSchottky} using different techniques. 

\medskip\noindent
\emph{Historical remarks:} Previously this subject was studied primarily in the case when $G=\PSL_2(\C)$. 
Minsky \cite{Minsky} showed that the space $PS(F_r)$ of primitive-stable representations of a free group $F_r$ into $\PSL_2(\C)$ is a domain of discontinuity for the action of $\Out(F_r)$ on $\mathbb{X}(F_r, \PSL_2(\C))$ properly containing the space of convex cocompact representations. 
This was the first example where $\Xa$ was shown to fail to be a maximal domain of discontinuity. Roughly, a representation $\rho\in\mathbb{X}(F_r,\PSL_2(\C))$ is primitive-stable 
if the restriction of the associated orbit map of the Cayley graph of $F_r$ into $\mathbb H^3$ is uniformly quasi-isometric on the axis of any primitive element of $F_r$.
Lee \cite{lee-ibundle,lee-compbody} found a domain of discontinuity for the action of $\Out(\Gamma)$ on $\mathbb{X}(\Gamma,\PSL_2(\C))$ that properly
contains ${\mathbb X}_{\mathrm{A}}(\Gamma,\PSL_2(\C))$ when $\Gamma$ is 
either the fundamental group of a closed nonorientable surface or a nontrivial free product of free groups and surface groups.

If $M$ is a compact hyperbolizable 3-manifold for which $\pi_1(M)$ is one-ended and not a surface group, then Canary and Storm \cite{Canary-Storm} again exhibited a domain of discontinuity for the action of $\Out(\pi_1(M))$ on $\mathbb{X}(\pi_1(M), \PSL_2(\C))$ that strictly contains $\mathbb{X}_{\mathrm{A}}(\pi_1(M), \PSL_2(\C))$.
However, Goldman \cite{goldman} conjectured that if $\Gamma$ is a closed orientable surface group, then quasifuchsian space $\mathbb{X}_{\mathrm{A}}(\Gamma, \PSL_2(\C))$ is a maximal domain of discontinuity for the action of $\Out(\Gamma)$ on $\mathbb{X}(\Gamma, \PSL_2(\C))$. It seems likely that if $\Gamma$ is not a closed orientable surface group, then $\mathbb{X}_{\mathrm{A}}(\Gamma, \PSL_2(\C))$ is not a maximal domain of discontinuity.

Our work generalizes the results and methods of Canary--Storm \cite{Canary-Storm} to the setting of representations of a word hyperbolic group $\Gamma$ into a semisimple Lie
group $G$ with finite center. The techniques are quite different than those used by Minsky \cite{Minsky} and Lee \cite{lee-compbody}, since the outer automorphism
group of a one-ended hyperbolic group and $\Out(F_r)$ can be structurally quite different. However, one sees a common link in the fact that amalgam Anosov
representations have the property that the restrictions of the orbit map of the Cayley graph of $\Gamma$ into the symmetric space associated to $G$ to
the Fuchsian vertex groups and to the registering subgroups are quasi-isometric embeddings. Therefore both amalgam Anosov and primitive-stable representations
are quasi-isometric on subsets of the Cayley graph which are large enough to establish proper discontinuity for the action, but not big enough
to guarantee that the representations are quasi-isometries on the entire Cayley graph.
In higher rank semisimple Lie groups the set of representations
that are quasi-isometric embeddings need not be open (see Gu\'eritaud--Guichard--Kassel--Wienhard \cite[Appendix A]{GGKW}), so if one attempted to
define an analogue of amalgam Anosov representations using only the language of quasi-isometric embeddings, one would likely
not obtain an open set. However, one could use recent work of Kapovich--Leeb--Porti \cite{KLP}
to give an alternate definition in the language of Morse embeddings of subgroups.

\medskip\noindent
\emph{Outline of the paper:}
In Section \ref{sec:anosov}, we give the definition of an Anosov representation and recall some of their basic properties.
In Section \ref{sec:jsj} we describe JSJ decompositions of one-ended hyperbolic groups and the associated structure theory of their outer automorphism groups, both originally due to Sela \cite{Sela}. In Section \ref{aareps},
we formally define amalgam Anosov representations and develop their basic properties. In Section \ref{mainproof}, we prove the main theorem. In Section \ref{sec:saa} we briefly introduce \emph{strongly} amalgam Anosov representations.
In Section \ref{sec:ex} we exhibit examples where $\Xa$ is a proper subset of $\Xaa$, examples of amalgam Anosov representations that are not strongly amalgam Anosov, and examples where $\Xaa$ is not a maximal domain of discontinuity. In the appendix, we establish Theorem \ref{main-appendix}.

\medskip\noindent
{\bf Acknowledgements:} The authors thank Fanny Kassel for conversations related to the appendix and
her work with Gu\'eritaud, Guichard, and Wienhard, and for suggesting that Lemma \ref{lem:ThanksFanny} should be true. The authors also thank the referee for many helpful comments on the original version of the manuscript.

\medskip

This material is based upon work supported by the National Science Foundation under Grant Numbers DMS-1045119, DMS-130692, and DMS-1361000. The authors also acknowledge support from the GEAR network (NSF grants DMS-1107452, DMS-1107263, and DMS-1107367).

\section{Anosov representations}\label{sec:anosov}

In this section, we recall the definition and basic properties of Anosov representations, see Labourie \cite{Labourie,Labourie2} and Guichard--Wienhard \cite{Guichard-Wienhard} for further details. We finish the section with a brief discussion of the special case of projective Anosov representations.

\subsection{Definitions}

We first recall that Gromov \cite{Gromov} defined a \emph{geodesic flow} $(U_0\Gamma,\{\phi_t\})$ associated with a hyperbolic group $\Gamma$. See Champetier \cite{Champetier} or Mineyev \cite[Thm.\ 60]{Mineyev} for further details. We follow Mineyev \cite{Mineyev}, since his definition gives properties best suited to our applications.

Let $\partial_\infty \Gamma^{(2)}$ be the space of pairs of distinct points on the Gromov boundary $\partial_\infty \Gamma$ of $\Gamma$. Mineyev \cite{Mineyev} showed that there exists a flow $\{\tilde\phi_t\}$ on
\[
\widetilde{U_0\Gamma} = \partial_\infty \Gamma^{(2)}\times\R,
\]
where $\R$ acts by translation on the real factor, and a metric on $\widetilde{U_0\Gamma}$ such that $\{\tilde\phi_t\}$ satisfies the following properties:
\begin{enumerate}

\item
$\Gamma$ acts cocompactly and by isometries on $\widetilde{U_0\Gamma}$ so that $\gamma\in\Gamma$ takes the leaf $(z^+,z^-)\times\R$ to the leaf $(\gamma(z^+),\gamma(z^-))\times \R$;

\item
the action of $\R$ by translation along the flow is bi-Lipschitz;

\item
the actions of $\Gamma$ and $\R$ commute;

\item
$t \mapsto \tilde\phi_t(y)$ is an isometric embedding of $\R$ into $\widetilde{U_0\Gamma}$ for all $y \in \widetilde{U_0\Gamma}$.

\end{enumerate}
Thus the flow descends to a flow $\{\phi_t\}$ on $U_0\Gamma=\widetilde{U_0\Gamma} / \Gamma$. While the construction depends on many choices, any flow space with the above properties will suffice to serve as the geodesic flow in the definition of an Anosov representation.
For example, if $\Gamma$ is the fundamental group of a closed negatively curved manifold $M$, then one may take $\widetilde{U_0\Gamma}$
to be the geodesic flow on the unit tangent bundle $T^1M$.

Let $G$ be a real semisimple Lie group with finite center, $(P^+, P^-)$ a pair of opposite parabolic subgroups of $G$, and $\mathcal{F}^\pm$ the (compact) flag variety $G / P^\pm$. Consider the Levi subgroup $L = P^+ \cap P^-$ associated with $(P^+, P^-)$, and let \hbox{$X = X(P^+, P^-)$ }be the space $G / L$, considered as an open subspace of \hbox{$\mathcal{F}^+ \times \mathcal{F}^-$}. This naturally equips $X$ with a pair $E^\pm$ of transverse distributions. If $\rho:\Gamma\to G$ is a representation, let $\widetilde {\mathcal{X}_\rho}=\widetilde{U_0\Gamma}\times X$ be the trivial $X$-bundle over $\widetilde{U_0\Gamma}$. Then $\Gamma$ acts on $\widetilde {\mathcal{X}_\rho}$ so that \hbox{$\gamma(y,x)=(\gamma(y),\rho(\gamma)(x))$} for all $\gamma\in\Gamma$. The quotient 
\[
\mathcal{X}_\rho=\Gamma\backslash\widetilde {\mathcal{X}_\rho}
\]
is a bundle over $U_0\Gamma$ with fiber $X$. The distributions $E^\pm$ on $X$ induce associated distributions, also called $E^\pm$, on $\mathcal{X}_\rho$.

A representation $\rho : \Gamma \to G$ is $(P^+, P^-)$-\emph{Anosov} if there exists a continuous section $\sigma : U_0\Gamma \to \mathcal{X}_\rho$ such that:
\begin{enumerate}[(i)]

\item
The section $\sigma$ is flat along $\R$-orbits.

\item
The (lifted) action of $\phi_t$ on the pullback $\sigma^* E^+$ (resp.\ $\sigma^* E^-$) dilates (resp.\ contracts).

\end{enumerate}
A section satisfying the above conditions will be called an \emph{Anosov section}. 

\subsection{Basic properties}

We say that $\rho : \Gamma \to G$ is an {\em Anosov} representation if there exist a pair $P^\pm$ of proper opposite parabolic subgroups of $G$ such that $\rho$ is $(P^+,P^-)$-Anosov with respect to the pair $P^\pm$.
The following summarizes the key properties of Anosov representations established by Labourie \cite{Labourie} and Guichard--Wienhard \cite{Guichard-Wienhard}.

\begin{Theorem}\label{thm:Guichard-Wienhard}
Let $\Gamma$ be a torsion-free hyperbolic group and $G$ be a semisimple Lie group with finite center.
\begin{enumerate}

\item
The set $\Xa$ of Anosov representations of $\Gamma$ into $G$ is open in $\X$.

\item
Every representation $\rho \in \Xa$ is discrete and faithful.

\item
Every representation $\rho \in \Xa$ is a quasi-isometric embedding (with respect to any word metric on $\Gamma$ and any left-invariant Riemannian metric on $G$).

\item
If $G$ has real rank $1$, then $\rho\in\X$ is Anosov if and only if it is convex cocompact.

\item
Every $\rho \in \Xa$ is well-displacing.

\end{enumerate}
\end{Theorem}

Recall that a representation is \emph{well-displacing} if there exists $(J,B)$ so that
\[
\frac{1}{J}||\gamma||-B\le \ell_\rho(\gamma)\le J ||\gamma||+B,
\]
where $||\gamma||$ denotes the translation length of $\gamma$ on the Cayley graph of $\Gamma$, with respect to some fixed generating set, and $\ell_\rho(\gamma)$ denotes the minimal translation length of the action of $\rho(\gamma)$ on $G$. 

The fact that Anosov representations are well-displacing (in fact, uniformly well-displacing on compact subsets of $\Xa$) can be used to show that $\Out(\Gamma)$ acts properly discontinuously on $\Xa$. Labourie \cite{Labourie2} proved this in the setting of the Hitchin component of representations of a closed surface group into $\PSL_n(\R)$, while Guichard--Wienhard \cite{Guichard-Wienhard} established it in full generality (see also Canary \cite{canary-survey}). Recall that the action of ${\rm Aut}(\Gamma)$ on ${\rm Hom}(\Gamma,G)$ given by
$\phi(\rho)=\rho\circ\phi^{-1}$ (for all $\phi\in {\rm Aut}(\Gamma)$ and $\rho\in {\rm Hom}(\Gamma,G)$) descends to an
action of ${\rm Out}(\Gamma)$ on $\X$.

In fact, we will need the following strengthening of the fact that $\Out(\Gamma)$ acts properly discontinuously on $\Xa$.
We recall that a sequence $\{A_n\}_{n\in\mathbb N}$ of subsets of a topological space $Y$ is said to \emph{exit every compact subset of} $X$ if given any compact subset $R$ of $Y$, there exists $N$ such that $A_n \cap R = \emptyset$ for $n\ge N$. 
Notice that $\Out(\Gamma)$ acts properly discontinuously on $\Xa$ if and only if whenever $\{\alpha_j\}$ is a sequence of distinct elements 
in $\Out(\Gamma)$ and $R$ is a compact subset of $\Xa$, then $\{\alpha_j(R)\}$ exits every compact subset of $\Xa$.

\begin{Theorem}\label{thm:anosovpd}
{\rm (Guichard--Wienhard \cite[Cor.\ 5.4]{Guichard-Wienhard})}
If $\Gamma$ is a torsion-free word hyperbolic group, then $\Out(\Gamma)$ acts properly discontinuously on $\Xa$. Moreover, if $R$ is a compact set in $\Xa$ and $\{\alpha_j\}$ is a sequence of distinct elements in $\Out(\Gamma),$ then $\{\alpha_j(R)\}$ exits every compact subset of 
the entire representation space $\X$.
\end{Theorem}

\begin{Remark}
Although the referenced statements only give proper discontinuity, it is easy to see that the proof yields the exiting behavior. We sketch the argument (see Canary \cite[Thm.\ 6.2]{canary-survey}). If $R$ is a compact subset of $\Xa$, then there exists $(J,B)$ such that every $\rho\in R$ is $(J,B)$-well-displacing. If $\{\phi_n\}$ is a sequence of distinct elements of $\Out(\Gamma)$, then there exists $\beta\in \Gamma$ so that $||\phi_n^{-1}(\beta)||\to\infty$. Therefore, $\inf_{\rho\in R}\ell_\rho(\phi_n^{-1}(\beta))\to\infty$, 
which implies that $\phi_n(R)$ exits every compact subset of $\X$.
\end{Remark}

We also observe below that the restriction of an Anosov representation to a hyperbolic subgroup is Anosov
if and only if the subgroup is quasiconvex.

\begin{Lemma}\label{lem:qconvexAnosov}
Suppose that $\Gamma$ is a word hyperbolic group, $G$ is a semisimple Lie group with finite center, and $P^\pm$ is a pair of proper opposite parabolic subgroups. If $\rho : \Gamma \to G$ is a $(P^+,P^-)$-Anosov representation and $\Lambda$ is a word hyperbolic subgroup of $\Gamma$, then $\rho |_\Lambda$ is $(P^+,P^-)$-Anosov
if and only if $\Lambda$ is a quasiconvex subgroup of $\Gamma$.
\end{Lemma}

\begin{proof}
We first suppose that $\Lambda$ is quasiconvex. There is a continuous injection $\eta : \partial_\infty \Lambda \to \partial_\infty \Gamma$, called the \emph{Cannon--Thurston map}. Consider the subset $\widetilde{U_0\Lambda}=\partial\Lambda^{(2)}\times\R$ of $\widetilde{U_0\Gamma}$. The action of $\Gamma$ on $ \widetilde{U_0\Gamma}$ restricts to an action of $\Lambda$ on $ \widetilde{U_0\Lambda}$ with properties (1)-(4) of Section 2.1. We can therefore regard $U_0\Lambda= \widetilde{U_0\Lambda}/\Lambda$ as the geodesic flow of $\Lambda$. 

Let $\sigma:U_0\Gamma\to \mathcal{X}_\rho$ be an Anosov section. Then $\sigma$ lifts to a section $\tilde\sigma:\widetilde{U_0\Gamma}\to\widetilde{\mathcal{X}_\rho}$. One then checks that the restriction of $\tilde\sigma$ to $\widetilde{U_0\Lambda}$ descends to an Anosov section of $\mathcal{X}_{\rho|_\Lambda}$. Therefore, $\rho|_\Lambda$ is \hbox{$(P^+,P^-)$-Anosov}.

Conversely, suppose that $\rho_\Lambda$ is $(P^+,P^-)$-Anosov. Since $\rho$ quasi-isometrically embeds both $\Lambda$ and $\Gamma$ into $G$, it follows
that $\Lambda$ is quasi-isometrically embedded in $\Gamma$, hence is quasiconvex.
\end{proof}

\subsection{Projective Anosov representations}

If $G= \SL_d(\R)$, $P^+$ is the stabilizer of a line in $\R^d$ and $P^-$ is the stabilizer of a complementary hyperplane, 
then a $(P^+,P^-)$-Anosov representation $\rho:\Gamma\to \SL_d(\R)$ is 
called \emph{projective Anosov}. 

If $\rho:\Gamma\to \SL_d(\R)$ is a projective Anosov representation, then 
\[
G/L\subset G/P^+\times G/P^- =\RPd\times\RPddual
\]
and the associated section
$\sigma:U_0\Gamma\to \mathcal{X}_\rho$ lifts to a section 
\[
\tilde\sigma:\widetilde{U_0\Gamma}\to \widetilde{X}_\rho=\widetilde{U_0\Gamma}\times G/L\subset \widetilde{U_0\Gamma}\times\RPd\times\RPddual.
\]
The defining properties of an Anosov section imply that
\[
\tilde\sigma (x,y,t)=((x,y,t),\xi(x),\theta(y))
\]
where
\[
\xi:\partial\Gamma\to \RPd\quad\textrm{and}\quad \theta:\partial\Gamma\to\RPddual.
\]
are $\rho$-equivariant continuous injective maps.
Moreover, since $\tilde\sigma(\widetilde{U_0\Gamma})\subset \widetilde{U_0\Gamma}\times G/L$, the maps $\xi$ and $\theta$ are \emph{transverse}, i.e., if $x$ and $y$ are distinct points in $\partial\Gamma$, then $\xi(x)$ and $\theta(y)$ span $\R^d$. Notice that here we are making use of the canonical identification of $\RPddual$ with the Grassmannian of hyperplanes in $\R^d$. (See Guichard--Wienhard \cite[Sec. 2.2]{Guichard-Wienhard} for details.)

In many situations, the existence of transverse $\rho$-equivariant limit maps is enough to guarantee that $\rho$ is projective Anosov. For example, we have the following.

\begin{Proposition}{\rm (Guichard--Wienhard \cite[Prop.\ 4.10]{Guichard-Wienhard})}
\label{convex irreducible}
If $\Gamma$ is a word hyperbolic group, $\rho:\Gamma\to \SL_d(\R)$ is an irreducible representation, and there exist transverse $\rho$-equivariant continuous injections $\xi:\partial\Gamma\to\RPd$ and $\theta:\partial\Gamma\to\RPddual$, then $\rho$ is projective Anosov.
\end{Proposition}

Moreover, Guichard and Wienhard \cite[Prop.\ 4.3 and Rmk. 4.12]{Guichard-Wienhard} showed that one can often use a \emph{Pl\"ucker representation} to reduce problems about general Anosov representations to problems about projective Anosov representations (see also Bridgeman--Canary--Labourie--Sambarino \cite[Cor. 2.13]{BCLS}).

\begin{Proposition}[Guichard--Wienhard \cite{Guichard-Wienhard}]\label{Plucker}
If $G$ is a semisimple Lie group with finite center and $P^\pm$ is a pair of proper opposite parabolic subgroups, then there
exists an integer $d=d(G, P^\pm)$ and an irreducible representation
\[
\tau : G \to \SL_{d}(\R)
\]
such that if $\Gamma$ is word hyperbolic, then $\rho:\Gamma\to G$ is $(P^+,P^-)$-Anosov if and only if $\tau\circ\rho$ is projective Anosov.
\end{Proposition}


\section{JSJ decompositions}\label{sec:jsj}

In this section, we recall the JSJ decomposition of a one-ended, torsion-free word hyperbolic group $\Gamma$ 
and the associated structure theory for
$\Out(\Gamma)$. We refer the reader to Gromov \cite{Gromov} and Bridson--Haefliger \cite{Bridson-Haefliger} for basics on hyperbolic groups and
to Serre \cite{Serre} for the theory of graphs of groups. 
Recall that a torsion-free word hyperbolic group is one-ended if and only if it is freely indecomposable but not cyclic.

If $\Gamma$ is the fundamental group of a finite graph of groups $\mathcal{G}(V, E)$ and every edge group is a cyclic subgroup of $\Gamma$, then $\mathcal{G}(V, E)$ is a $\Z$-\emph{splitting} of $\Gamma$. A vertex group $\Gamma_v$ of $\mathcal{G}(V, E)$ is \emph{Fuchsian} if $\Gamma_v$ is not cyclic and there exists a compact surface $F_v$ with boundary $\partial F_v$ such that $\Gamma_v \cong \pi_1(F_v)$ and the edge groups adjacent to $v$ are exactly the (conjugacy classes of) cyclic subgroups of $\Gamma_v$ associated with components of $\partial F_v$. A vertex group $\Gamma_v$ is \emph{rigid} if it is not infinite cyclic or Fuchsian and does not admit a $\Z$-splitting for which every edge group adjacent to $v$ is conjugate into a vertex group in the $\Z$-splitting of $\Gamma_v$.

Sela \cite{Sela} developed a canonical JSJ decomposition for one-ended word hyperbolic groups. 
We use a version due to Bowditch \cite{Bowditch}, which essentially agrees with Sela's in the torsion-free case and provides slightly stronger uniqueness results. 
The following theorem records the key properties of the JSJ splitting.

\begin{Theorem}{\rm (Bowditch \cite{Bowditch}, Sela \cite{Sela})}
\label{thm:jsj}
If $\Gamma$ is a torsion-free, one-ended, indecomposable hyperbolic group, there exists a canonical $\Z$-splitting $\Gamma \cong \pi_1(\mathcal{G}(V, E))$ of $\Gamma$,
called the JSJ-splitting, where each vertex group is either a maximal infinite cyclic subgroup of $\Gamma$, a maximal Fuchsian subgroup, or a rigid vertex group, such that:
\begin{enumerate}

\item
No two vertices of the same type are adjacent and no Fuchsian vertex is adjacent to a rigid vertex.

\item
Every vertex group is a quasiconvex subgroup of $\Gamma$.

\item
If $H$ is an infinite cyclic subgroup of $\Gamma$ and $\Gamma$ splits over $H$, then $H$ can be conjugated into an infinite cyclic or Fuchsian vertex group.

\item
If $\phi\in {\rm Aut}(\Gamma)$ and $v\in V$, then there exists $w\in V$ so that $\phi(\Gamma_v)$ is conjugate to $\Gamma_w$.
\end{enumerate}
\end{Theorem}

\noindent
{\bf Remark:} Any graph of groups decomposition of $\Gamma$ gives rise to a Bass--Serre tree.
Bowditch's JSJ splitting is canonical in the sense that the Bass--Serre trees 
associated to any two JSJ splittings are $\Gamma$-equivariantly isomorphic.
In particular, the collection of vertex groups is well-defined up to conjugacy. 
In the language of Bass--Serre trees, one can strengthen (4) to say that if $\phi\in {\rm Aut}(\Gamma)$, then there is an automorphism 
$H_\phi$ of the Bass--Serre tree that conjugates the action of $\Gamma$ to the action twisted by $\phi$, i.e., the action of $\phi(\gamma)H_\phi$ agrees with the action of $H_\phi\gamma$ for all $\gamma\in\Gamma$.

\medskip

Sela \cite{Sela}, see also Levitt \cite{levitt}, used the JSJ splitting of $\Gamma$ to obtain a decomposition of a finite index subgroup of $\Out(\Gamma)$.
Let $\Out_1(\Gamma)$ denote the finite index subgroup of $\Out(\Gamma)$ consisting of all elements $\phi$ such that if $v\in V$, then $\phi(\Gamma_v)$ is
conjugate to $\Gamma_v$. Notice that this property is invariant under inner automorphisms, so it makes sense in $\Out(\Gamma)$.

If $\Delta$ is a subgroup of $\Gamma$, then there is a restriction map
\[
r_\Delta:\X\to \mathbb{X}(\Delta,G)
\]
where $r_\Delta(\rho)=\rho|_\Delta$.
If $v\in V$, the normalizer of $\Gamma_v$ in $\Gamma$ is exactly $\Gamma_v$, which follows immediately from 
\cite[Prop.\ $\mathrm{III}.\Gamma.16$]{Bridson-Haefliger} if
$\Gamma_v$ is  cyclic  and from \cite[Lem.\ 2.1]{carette} otherwise.
Therefore, we obtain a well-defined map
\[
\psi_v:\Out_1(\Gamma)\to \Out(\Gamma_v),
\]
such that 
\[
r_{\Gamma_v}(\rho\circ\phi)=r_{\Gamma_v}(\rho)\circ \psi_v(\phi)
\]
for all $\rho\in\X$ and $\phi\in\Out_1(\Gamma)$. Note that $\rho\circ\phi=\phi^{-1}(\rho)$ is well-defined since inner automorphisms
of $\Gamma$ act trivially on $\X$.

If $\Gamma_v$ is a Fuchsian vertex group, let $\Mod_0(F_v)$ be the group of isotopy classes of homeomorphisms of $F_v$ whose restriction to $\partial F_v$ is homotopic to the identity. Equivalently, this is the subgroup of $\Out(\Gamma_v)$ consisting of outer automorphisms preserving the conjugacy class of each peripheral element. Every element of $\Mod_0(F_v)$ extends, though not uniquely, to an outer automorphism of $\Gamma$.

Let $\Out_0(\Gamma)$ denote the subgroup of $\Out_1(\Gamma)$ consisting of all $\phi\in\Out_1(\Gamma)$ such that $\psi_v(\phi)$ is trivial if $\Gamma_v$ is rigid or cyclic and $\psi_v(\phi)\in {\rm Mod}_0(F_v)$
if $\Gamma_v$ is Fuchsian. Define 
\[
{\rm Mod}_F(\Gamma)=\oplus_{v\in V_f} {\rm Mod}_0(F_v),
\]
where $V_f$ is the collection of Fuchsian vertices,
and consider the homomorphism
\[
\psi=\oplus_{v\in V_f} \psi_v:\Out_0(\Gamma)\to {\rm Mod}_F(\Gamma).
\]

The kernel of $\psi$ is generated by twists about cyclic vertex groups (see Levitt \cite[Prop.\ 2.2]{levitt}).
If $\Gamma_v$ is a cyclic vertex group and $e$ is an edge with $v$ as an endpoint, then
there is an associated amalgamated free product $\Gamma=A*_{\Gamma_e}B$ (if $e$ is separating) or HNN extension $\Gamma=A*_{\Gamma_e}$ (if $e$
does not separate) and $\Gamma_v\subset A$. 
If $z\in\Gamma_v$, we define the \emph{twist} $D_{z,e}$ as the image in $\Out(\Gamma)$ of the automorphism that acts trivially on $A$ and by conjugation by $z$ on $B$ in the amalgamated
free product case, and acts by taking the stable letter $t$ to $zt$ in the HNN case. Notice that this definition is motivated by the group-theoretical reformulation of a Dehn twist about a simple closed curve on a surface. If $v$ is a cyclic vertex, then the {\em characteristic twist subgroup} $\mathcal{T}_v$ is
the group generated by all twists about the cyclic vertex group $\Gamma_v$. 

Let $\mathcal{T}$ be the subgroup of $\Out_0(\Gamma)$ generated by all characteristic twist groups.
The group $\mathcal{T}$ is free abelian. Moreover,
\[
\mathcal{T}=\bigoplus_{v\in V_c} \mathcal T_{v},
\]
where $V_c$ is the collection of cyclic vertices and each $\mathcal{T}_v$ is a free abelian group of rank $n_v-1$, where $n_v$ is the number of edges adjacent to $v$
(see Levitt \cite[Prop.\ 3.1, Thm.\ 5.1]{levitt}).
For each $v\in V_c$, let 
\[
p_{v}:\mathcal{T}\to \mathcal T_{v}
\]
be the obvious projection map.

Sela proved that $\Out_0(\Gamma)$ has finite index in $\Out(\Gamma)$ and is a central extension of ${\rm Mod}_F(\Gamma)$
by $\mathcal{T}$.

\begin{Theorem}\label{thm:selaout}
{\rm (Levitt \cite[Thm.\ 5.1]{levitt}, Sela \cite[Thm.\ 1.9(ii)]{Sela})}
Let $\Gamma$ be a one-ended, torsion-free hyperbolic group with JSJ splitting $\mathcal{G}(V, E)$. Then, $\Out_0(\Gamma)$ is a finite
index subgroup of $\Out(\Gamma)$ and $\psi:\Out_0(\Gamma)\to {\rm Mod}_F(\Gamma)$ gives rise to a central extension
\[
\{1\} \rightarrow \mathcal{T} \rightarrow \Out_0(\Gamma)\xrightarrow{\;\; \psi \; \; } \Mod_F(\Gamma) \rightarrow \{1\}.
\]
\end{Theorem}

\section{Basic properties of amalgam Anosov representations}
\label{aareps}

We are almost ready for the precise definition of an amalgam Anosov representation. It only remains
to carefully define what it means to register a cyclic vertex group. Suppose that $\Gamma_v$ is a cyclic vertex group and $\{w_1,\ldots,w_{n_v}\}$ are the vertices of 
$\mathcal{G}(V,E)$ adjacent to $v$. We say that a free subgroup $H$ of $\Gamma$ \emph{registers} $\Gamma_v$ 
if $H$ is freely generated by nontrivial elements $a_0\in \Gamma_v$ and $\{a_1,\ldots,a_{n_v}\}$ where $a_i\in \Gamma_{w_i}$ for each $i$.
We will see that this implies that a finite index subgroup of the twist group $\mathcal{T}_v$ preserves $H$ (up to conjugacy) and embeds in $\Out(H)$.

\begin{Definition}
\label{aadef}
If $\Gamma$ is a one-ended word hyperbolic group and $G$ is a semisimple Lie group, we say that a representation
$\rho\in\X$ is {\em amalgam Anosov} if it satisfies the following two conditions:
\begin{enumerate} 
\item
if $\Gamma_v$ is a Fuchsian vertex group, then $\rho|_{\Gamma_v}$ is Anosov, and
\item
if $\Gamma_v$ is a cyclic vertex group, then there exists a free subgroup $H$ of $\Gamma$ that registers $\Gamma_v$ such that $\rho|_H$ is Anosov.
\end{enumerate}
\end{Definition}

We emphasize that  the restricted representations of the form $\rho|_{\Gamma_v}$ or $\rho|_H$ in the above definition
need not all be Anosov with respect to the same pair of opposite parabolic subgroups of $G$. We will see that  the space $\Xaa$ of amalgam Anosov representations of $\Gamma$ into $G$
is an open, $\Out(\Gamma)$-invariant subset of $\X$ and that Anosov representations are amalgam Anosov, 
i.e., $\Xa\subset \Xaa$.

It will be convenient to introduce some notation. Let 
\[
\mathbb X_F(\Gamma,G)=\bigcap_{v\in V_f} r_{\Gamma_v}^{-1}({\mathbb X}_{\mathrm{A}}(\Gamma_v,G)),
\]
where $V_f$ is the collection of Fuchsian vertices and
\hbox{$r_{\Gamma_v}:\X\to\mathbb{X}(\Gamma_v, G)$} is the restriction map for each $v\in V_f$.
Similarly, if $v\in V_c$ is a cyclic vertex, let  
\[
\mathbb{X}_v(\Gamma,G)=\bigcup_{H\ {\rm registers}\ \Gamma_v} r_H^{-1}(\mathbb{X}_{\mathrm{A}}(H,G))
\]
and
\[
\mathbb{X}_C(\Gamma,G)=\bigcap_{v\in V_c} \mathbb X_v(\Gamma_v,G).
\]
Notice that $\Xaa=\mathbb{X}_C(\Gamma,G)\cap \mathbb{X}_F(\Gamma,G)$.

\begin{Proposition}
\label{AareAA}
If $\Gamma$ is a one-ended, torsion-free word hyperbolic group and $G$ is a semisimple Lie group with finite center,
$\Xaa$ is an open, $\Out(\Gamma)$-invariant subset of $\X$ and 
\[
\Xa \subset \Xaa.
\]
Moreover, $\mathbb{X}_C(\Gamma, G)$ and $\mathbb{X}_F(\Gamma, G)$ are also open $\Out(\Gamma)$ invariant subsets of $\X$.
\end{Proposition}

\begin{proof}
Since each $r_{\Gamma_v}$ is continuous, each $\mathbb{X}_{\mathrm{A}}(\Gamma_v,G)$ is open in $\mathbb{X}(\Gamma_v,G)$,
and there are only finitely many Fuchsian vertex groups,  $\mathbb{X}_F(\Gamma,G)$ is open in $\X$. 
If $v$ is a Fuchsian vertex and $\phi\in \Out(\Gamma)$, then there exists a Fuchsian vertex $w$ so that $\phi(\Gamma_v)$
 is conjugate to $\Gamma_w$. 
It follows, since $\phi(\rho)=\rho\circ\phi^{-1}$, that 
\[
\phi(r_{\Gamma_v}^{-1}({\mathbb X}_{\mathrm{A}}(\Gamma_v,G)))=r_{\Gamma_w}^{-1}({\mathbb X}_{\mathrm{A}}(\Gamma_w,G)).
\]
Therefore, $\mathbb{X}_F(\Gamma,G)$ is $\Out(\Gamma)$-invariant.

If $H$ registers a cyclic vertex group $\Gamma_v$, then
since $\mathbb{X}_{\mathrm{A}}(H, G)$ is open in $\mathbb{X}(H, G)$ and $r_H$ is continuous, 
$r_H^{-1}(\mathbb{X}_{\mathrm{A}}(H,G))$ is open in $\X$, so $\mathbb{X}_v(\Gamma,G)$ is an open subset of $\X$.
Since there are only finitely many cyclic vertex groups, $\mathbb{X}_C(\Gamma,G)$ is open in $\X$.
If $\phi\in\Out(\Gamma),$ then $\phi$ permutes the cyclic vertex groups 
and takes registering subgroups to registering subgroups (up to conjugacy), so $\mathbb{X}_C(\Gamma,G)$ is $\Out(\Gamma)$-invariant.

We established that $\mathbb{X}_C(\Gamma,G)$ and $\mathbb{X}_F(\Gamma,G)$ are open and $\Out(\Gamma)$-invariant in $\X$, so their intersection $\Xaa$ is as well.

We will need the following lemma in the proof that $\Xa\subset\Xaa$.

\begin{Lemma}\label{lem:RegistersExist}
If $\Gamma$ is a one-ended, torsion-free word hyperbolic group and 
$\Gamma_v$ is a cyclic vertex group of its JSJ splitting, there exists a quasiconvex subgroup $H \subset \Gamma$ that registers $\Gamma_v$.
\end{Lemma}

\begin{proof}
Choose a generator $a$ for $\Gamma_v$, and for each vertex $w_i$ adjacent to $v$ by the edge $e_i$ let $g_i$ be an element of $\Gamma_{w_i}$, no power of which lies in the edge group $\Gamma_{e_i}$. Such an element $g_i$ exists because $\Gamma_{w_i}$ is not virtually cyclic.

Applying a ping-pong argument, (e.g., Bridson--Haefliger \cite[III.$\Gamma$.3]{Bridson-Haefliger}), one sees that there exists $N$ so that, for all $n\ge N$, the group $H_n$ generated by $\{a^n,g_1^n,\ldots,g_{n_v}^n\}$ is a free subgroup of $\Gamma$ freely generated by $\{a^n,g_1^n,\ldots,g_{n_v}^n\}$. Thus, $H_n$ is a registering subgroup for $T$. Moreover, one may choose $N$ large enough so that $H_n$ is quasiconvex for all $n \ge N$ (see Gitik \cite{Gitik}).
\end{proof}

Let $\rho : \Gamma \to G$ be Anosov with respect to the parabolic subgroups $P^\pm$ of $G$. 
If $\Gamma_v$ is a Fuchsian vertex group of the JSJ decomposition of $\Gamma$, then $\Gamma_v$ is a quasiconvex subgroup of $\Gamma$ (see Theorem \ref{thm:jsj}), so Lemma \ref{lem:qconvexAnosov} implies that $\rho|_{\Gamma_v}$ is Anosov. If $\Gamma_v$ is a cyclic vertex subgroup, then Lemma \ref{lem:RegistersExist} guarantees that there exists a quasiconvex subgroup $H$ of $\Gamma$ that registers $\Gamma_v$. Therefore, again by Lemma 
\ref{lem:qconvexAnosov}, $H$ registers $\Gamma_v$ and $\rho|_H$ is Anosov. It follows that $\rho$ is amalgam Anosov.
\end{proof}

\noindent
{\bf Remark:} The fact that $\Xa\subset\Xaa$ can also be obtained almost immediately by combining Theorem \ref{main-appendix}
and Lemma \ref{lem:qconvexAnosov}, but in the proof of Proposition \ref{AareAA} we give a more self-contained and elementary argument.

\medskip

When we construct examples it will be useful to observe that the Pl\"ucker representation interacts well with the definition 
of amalgam Anosov representations, especially when $G$ has real rank 1.

\begin{Lemma}
\label{Pluckeramalgam}
Let $G$ be a semisimple Lie group with finite center, $P^\pm$ be a pair of proper opposite parabolic subgroups of $G$ and $\tau:G\to \SL_d(\R)$ be the Pl\"ucker representation given by Proposition \ref{Plucker}. If $\Gamma$ is a one-ended, torsion-free hyperbolic group and $\rho:\Gamma\to G$ is a representation such that the restriction of $\rho$ to each Fuchsian vertex subgroup of $\Gamma$ is $(P^+,P^-)$-Anosov and for each cyclic vertex group $\Gamma_v$ of $\Gamma$ there is a registering subgroup $H$, such that $\rho|_H$ is $(P^+,P^-)$-Anosov, then $\tau\circ\rho$ is amalgam Anosov. 

In particular, if $G$ has rank 1 and $\rho:\Gamma\to G$ is amalgam Anosov, then $\tau\circ\rho$ is amalgam Anosov.
\end{Lemma}

\begin{proof}
Properties (1) and (2) in the definition of an amalgam Anosov representation
and Proposition \ref{Plucker} imply that the restriction of $\tau\circ\rho$ to any Fuchsian vertex group is projective Anosov. Further, it implies that for each cyclic vertex group 
$\Gamma_v$ of $\Gamma$ there is a subgroup $H$ registering $\Gamma_v$ such that $(\tau\circ\rho)|_H$ is projective Anosov. 
Therefore, $\tau\circ \rho$ is amalgam Anosov.

If $G$ has real rank 1, there is a unique conjugacy class of pairs of proper opposite parabolic subgroups. Consequently, every amalgam Anosov representation satisfies the assumptions of the lemma.
\end{proof}


\section{Proof of the Main Theorem}
\label{mainproof}
 
In this section we give the proof of our main result, Theorem \ref{thm:main}. Since $\Out_0(\Gamma)$ has finite index in $\Out(\Gamma)$, Theorem \ref{thm:main} 
follows from Proposition \ref{AareAA} and the following strong proper discontinuity result.

\begin{Proposition}
\label{strongPD}
If $R$ is a compact subset of $\Xaa$ and $\{\phi_n\}$ is a sequence of distinct elements of $\Out_0(\Gamma),$ then $\{\phi_n(R)\}$ exits every compact subset of the entire representation space $\X$.
\end{Proposition}

We break the proof of Proposition \ref{strongPD} into the analysis of sequences in the twist group $\mathcal{T}$ and sequences that project to a sequence
of distinct elements of 
${\rm Mod}_F(\Gamma)$.

The following lemma is a nearly immediate consequence of Theorem \ref{thm:anosovpd} and our definitions.

\begin{Lemma}\label{lem:mhf}
If $R$ is a compact subset of $\mathbb{X}_F(\Gamma, G)$ and $\{\phi_n\}$ is a sequence of elements in $\Out_0(\Gamma)$ 
such that $\{\psi(\phi_n)\}$ is a sequence of distinct elements of $\Mod_F(\Gamma)$, then $\{\phi_n(R)\}$ exits every compact subset of $\X$.
\end{Lemma}

\begin{proof}
If not, we can pass to a subsequence so that there exists $v\in V_f$ such that $\{\psi_v(\phi_n)\}$ is a sequence of distinct elements of $\Mod_0(F_v)$ 
and $\{\phi_n(R)\}$ does not exit every compact subset of $\X$. Then, $r_{\Gamma_v}(R)$ is a compact subset of $\mathbb{X}_{\mathrm{A}}(\Gamma_v,G)$ and 
$\{\psi_v(\phi_n)(r_{\Gamma_v}(R))\}$ does not exit every compact subset of $\mathbb{X}(\Gamma_v,G)$. However, this contradicts Theorem \ref{thm:anosovpd}.
\end{proof}

We now use Theorem \ref{thm:anosovpd} to study the action of $\mathcal{T}$ on $\mathbb X_C(\Gamma,G)$.

\begin{Lemma}\label{lem:characteristic}
If $R$ is a compact subset of $\mathbb{X}_C(\Gamma, G)$ and $\{\phi_n\}$ is a sequence of distinct elements of $\mathcal{T}$, 
then $\{\phi_n(R)\}$ exits every compact set of  $\X$.
\end{Lemma}

\begin{proof} 
Let $R$ be a compact subset of $\mathbb{X}_C(\Gamma, G)$. Suppose, for a contradiction, that there exists a sequence $\{\phi_n\}$ of distinct elements of $\mathcal{T}$ such that 
$\{\phi_n(R)\}$ does not exit every compact subset of $\X$. We can pass to a subsequence so that there exists a characteristic twist subgroup $\Gamma_v$ 
such that $\{p_v(\phi_n)\}$ is a sequence of distinct elements of $\Gamma_v$ and $\{\phi_n(R)\}$ still does not exit every compact subset of $\X$. Since $R$ is compact, $r_H^{-1}(\mathbb{X}_A(H,G))$ is open for each $H$, and 
$\X$ is locally compact, we can pass to a further subsequence so that there exists a compact subset $R_0$ of $R$ and a subgroup $H_0$
of $\Gamma$ that registers $\Gamma_v$ such that $R_0\subset r_{H_0}^{-1}(\mathbb{X}_{\mathrm{A}}(H_0,G))$ and $\{\phi_n(R_0)\}$ does not exit every compact subset of $\X$.

Let $a_0$ be the generator of $\Gamma_v$ that lies in the generating set $\{a_0,a_1,\ldots,a_{n_v}\}$ of $H_0$ and let 
$\mathcal{T}_{H_0}$ be the subgroup of $\mathcal{T}_v$ generated by twists of the form $D_{z,e}$ where $z$ is a power of $a_0$ and $e$ is an edge incident to $v$.
Since $\mathcal T_{H_0}$ has finite index in $\mathcal{T}_v$, we can pass to a further subsequence so that there exists $\phi_0\in \mathcal{T}_v$ with
\[
\phi_0\circ p_v(\phi_n)=p_v(\phi_0\circ\phi_n)\in \mathcal{T}_{H_0}
\]
for all $n$, where $p_v:\mathcal{T}\to\mathcal{T}_v$ is the projection map.

Notice that if $e_i$ is an edge connecting $v$ to $w_i$, then there is a representative of $D_{z,e_i}$ that preserves $H_0$ and acts on $H_0$ 
by taking $a_i$ to $za_iz^{-1}$ and taking every other generator to itself.
If $\{e_1,\ldots,e_{n_v}\}$ are the edges adjacent to $v$, then $\mathcal T_{H_0}$ is generated by $\{D_{a_0,e_1},\ldots,D_{a_0,e_{n_v-1}}\}$.
Therefore, there is an injective homomorphism
\[
s_{H_0}:\mathcal{T}_{H_0}\to \Out(H_0).
\]
so that 
\[
r_{H_0}(\rho\circ\phi)=r_{H_0}(\rho)\circ s_{H_0}(\phi)
\]
for all $\rho\in\X$ and $\phi\in \mathcal{T}_{H_0}$.
Moreover, if $\mathcal{T}_w$ is the characteristic twist subgroup for a different cyclic vertex group, then every element of $\mathcal{T}_w$ acts 
trivially on $H_0$ (up to conjugacy), since no cyclic vertices are adjacent. Therefore, $\mathcal{T}_0=p_v^{-1}(\mathcal{T}_{H_0})$ is a finite index subgroup
of $\mathcal{T}$ and $s_{H_0}$ extends to a map
\[
s_{0}:\mathcal{T}_{0}\to \Out(H_0)
\]
that is trivial on $\mathcal{T}_w$ if $w\in V_c-\{v\}$ and such that
\[
r_{H_0}(\rho\circ\phi)=r_{H_0}(\rho)\circ s_{0}(\phi)
\]
for all $\rho\in\X$ and $\phi\in \mathcal{T}_{0}$.

Since $\{s_{0}(\phi_0\circ\phi_n)\}$ is a sequence of distinct elements of $\Out(H_0)$, Theorem \ref{thm:anosovpd} implies that 
$\{s_{0}(\phi_0\circ\phi_n)(r_{H_0}(R_0))\}$ exits every compact subset of $\mathbb{X}(H_0,G)$. Therefore, $\{\phi_0(\phi_n(R_0))\}$ exits every compact subset of $\X$. Since $\phi_0$ induces a homeomorphism of $\X$, it follows that $\{\phi_n(R_0)\} = \{\phi_0^{-1}( \phi_0(\phi_n(R_0)))\}$ also exits every compact subset of $\X$. This contradiction completes the proof.
\end{proof}

We now combine Lemmas \ref{lem:mhf} and \ref{lem:characteristic} to establish Proposition \ref{strongPD}.

\medskip

{\em Proof of Proposition \ref{strongPD}:}
If the proposition fails, then there exists a compact subset $R$ of $\Xaa$ and a sequence $\{\phi_n\}$ of distinct elements of $\Out_0(\Gamma)$ such that 
$\{\phi_n(R)\}$ does not exit every compact subset of $\X$. After passing to a subsequence, we can assume that $\{\psi(\phi_n)\}$ is either
a sequence of distinct elements in $\Mod_F(\Gamma)$ or is constant. If $\{\psi(\phi_n)\}$ is a sequence of distinct elements, then Lemma \ref{lem:mhf} immediately implies that $\{\phi_n(R)\}$ leaves every compact subset of $\X$. If $\{\psi(\phi_n)\}$ is a constant sequence, then Theorem \ref{thm:selaout} implies that there exists a sequence $\{\beta_n\}$ of distinct elements of $\mathcal{T}$ such that $\phi_n = \phi_1 \circ \beta_n$ for all $n$. Lemma \ref{lem:characteristic} implies that $\{\beta_n(R)\}$ exits every compact subset of $\X$. Since $\phi_1$ induces a homeomorphism of $\X$, it follows that $\{\phi_n(R)\} = \{\phi_1 (\beta_n(R))\}$ also exits every compact subset of $\X$. This contradiction completes the proof of Proposition \ref{strongPD} and the proof of Theorem \ref{thm:main}.
\qed

\section{Strongly amalgam Anosov representations}\label{sec:saa}

Strongly amalgam Anosov representations are a natural and easy to construct class of amalgam Anosov 
representations.

\begin{Definition} 
\label{saadef}
If $\Gamma$ is a one-ended, torsion-free word hyperbolic group and $G$ is a semisimple Lie group with finite
center, we say that a representation $\rho:\Gamma\to G$ is {\em strongly amalgam Anosov} if there
exists a pair $P^\pm$ of opposite parabolic subgroups such that if $\Gamma_v$ is a Fuchsian or rigid vertex
group of the JSJ splitting of $\Gamma$, then $\rho|_{\Gamma_v}$ is $(P^+,P^-)$-Anosov.
\end{Definition}

Let $\Xsa\subset\X$ be the set of strongly amalgam Anosov representations.
In the appendix we will show that strongly amalgam Anosov representations are amalgam Anosov, see Theorem \ref{main-appendix}.
On the other hand, it is an immediate consequence of Lemma \ref{lem:qconvexAnosov} and the fact that vertex subgroups are quasiconvex that Anosov representations are strongly amalgam Anosov. Therefore,
\[
\Xa\subset\Xsa\subset\Xaa.
\]
Moreover, the proof of Proposition \ref{AareAA} may easily be adapted to show that $\Xsa$ is an open, $\Out(\Gamma)$-invariant
subset of $\X$. It then follows immediately from our main result, Theorem \ref{thm:main}, that $\Xsa$ is a domain of discontinuity for
the action of $\Out(\Gamma)$ on $\X$.

\begin{Corollary}
If $\Gamma$ is a one-ended word hyperbolic group and $G$ is a semisimple Lie group with finite center,
then $\Xsa$ is an open, $\Out(\Gamma)$-invariant subset of $\X$ and $\Out(\Gamma)$ acts properly discontinuously
on $\Xsa$.
\end{Corollary}

\section{Examples}\label{sec:ex}

Since 
\[
\Xa\subset\Xsa\subset\Xaa
\]
and each of these sets is a domain of discontinuity for the action of $\Out(\Gamma)$ on $\X$, it
is natural to ask whether these inclusions are proper and whether $\Xaa$ is a maximal domain
of discontinuity for $\Out(\Gamma)$.

In this section, we study families of examples exhibiting various phenomena. 
In Section \ref{rigidsubsec} we observe that if $\Gamma$ is rigid, then $\Xa=\Xsa\ne\Xaa=\X$.
In Section \ref{surfacegroups} we survey previous results and conjectures in the case where $\Gamma$
is a closed surface group. In Section \ref{3manifoldgroups}, we see that if $\Gamma$ is the fundamental group of a 
compact, hyperbolizable \hbox{3-manifold} with boundary,
but not a surface group, then $\mathbb X_{\rm A}(\Gamma,\PSL_2(\C))\ne\mathbb X_{{\rm A}^2}(\Gamma,\PSL_2(\C))$. 
In Section \ref{rankone}, we exhibit examples where $\Xa$ is empty,
but $\Xsa$ has positive dimension. In Section \ref{aanotsa}, we exhibit nonrigid groups $\Gamma$, where $\Xsa\ne\Xaa$.
Finally, in Section \ref{notmax},
we exhibit examples where $\Xaa$ is not a maximal domain of discontinuity for the action of $\Out(\Gamma)$ on $\X$.

\subsection{Rigid Groups}\label{rigidsubsec}

If a one-ended hyperbolic group $\Gamma$ has a trivial JSJ splitting then it is either rigid, i.e., admits no $\Z$-splitting, or is a surface group. If $\Gamma$ is rigid and nontrivial then every representation in $\X$ is  amalgam Anosov, while the trivial representation
fails to be either Anosov or strongly amalgam Anosov, 
so 
$$\Xa=\Xsa\ne\Xaa=\X.$$

On the other hand, if $\Gamma$ is not rigid, then the trivial representation is not amalgam Anosov,
so $\X\ne\Xaa$. Further, $\Out(\Gamma)$ is infinite and does not act properly discontinuously on all of $\X$, 
since it fixes the trivial representation.

\begin{Lemma}
\label{rigid case}
If $\Gamma$ is a one-ended, torsion-free hyperbolic group and $G$ is a semisimple Lie group with finite center, then $\X=\Xaa$ if and only if $\Gamma$ is rigid. 
Moreover, $\Out(\Gamma)$ acts properly discontinuously on $\X$ if and only if $\Gamma$ is rigid.
\end{Lemma}

\subsection{Surface groups}
\label{surfacegroups}

If $\Gamma$ is the fundamental group of a closed, oriented hyperbolic surface $S$, then
\[
\Xa=\Xsa=\Xaa,
\]
and $\mathbb{X}_{\mathrm{A}}(\pi_1(S),\PSL_2(\R))$ is the disjoint union of the Teichm\"uller spaces $\mathcal{T}(S) \cup \mathcal{T}(\bar S)$ of $S$ and $S$ with the opposite orientation. If $G=\PSL_2(\C)$, then $\mathbb{X}_{\mathrm{A}}(\pi_1(S),\PSL_2(\C))$ is the space of quasifuchsian representations and may be identified with $\mathcal{T}(S)\times\mathcal{T}(\bar S)$. Goldman made the following conjecture.

\begin{Conjecture}[Goldman \cite{goldman}]
If $S$ is a closed orientable hyperbolic surface and $G$ is $\PSL_2(\R)$ or $\PSL_2(\C)$, then $\mathbb{X}_{\mathrm{A}}(\pi_1(S),G)$ is a maximal domain of discontinuity for the action of $\Out(\pi_1(F))$ on $\mathbb{X}(\pi_1(S),G)$. Moreover, $\Out(\pi_1(S))$ acts ergodically on
\[
\mathbb{X}(\pi_1(S),\PSL_2(\C)) \smallsetminus \mathbb{X}_{\mathrm{A}}(\pi_1(S),\PSL_2(\C))
\]
and on each component of 
\[
\mathbb{X}(\pi_1(S),\PSL_2(\R)) \smallsetminus \mathbb{X}_{\mathrm{A}}(\pi_1(S),\PSL_2(\R))
\]
(unless $S$ has genus 2 and the component consists of representations with Euler number zero).
\end{Conjecture}

There has been some recent progress on this conjecture when $G=\PSL_2(\mathbb R)$. March\'e and Wolff \cite{marche-wolff} proved that if $S_2$ is the closed orientable surface of genus 2, then $\Out(\pi_1(S_2))$ acts ergodically on the two components of $\mathbb{X}(\pi_1(S_2),\PSL_2(\R))$ consisting of representations of Euler number $1$ and $-1$. Souto \cite{souto} 
subsequently proved that if $S$ does not have genus two, then $\Out(\pi_1(S))$ acts ergodically on the component consisting of representations with Euler number zero. In genus two, March\'e and Wolff showed that the set of nonelementary representations with Euler number zero is the disjoint union of two $\Out(\Gamma)$-invariant open sets on which the action is ergodic.

Lee \cite{lee-ibundle} proved that if $S$ is a closed orientable surface and $\Out(\pi_1(S))$ preserves and acts properly discontinuously on an open subset $U$ inside $\mathbb{X}(\pi_1(S),\PSL_2(\C))$, then $U$ cannot intersect $\partial \mathbb{X}_{\mathrm{A}}(\pi_1(S),\PSL_2(\C))$. 
On the other hand, 
Lee \cite{lee-ibundle} proved that if $S$ is a closed, nonorientable, hyperbolic surface, then $\mathbb{X}_{\mathrm{A}}(\pi_1(S),\PSL_2(\C))$ is not a maximal domain of discontinuity.

\begin{Theorem}[Lee \cite{lee-ibundle}]
\label{lee-nonmax}
If $S$ is a closed, nonorientable hyperbolic surface, then there is an open $\Out(\pi_1(S))$-invariant subset $W(S)$ of $\mathbb{X}(\pi_1(S),\PSL_2(\C))$ such that
\begin{enumerate}

\item
$\Out(\pi_1(S))$ acts properly discontinuously on $W(S)$,

\item
$\mathbb{X}_{\mathrm{A}}(\pi_1(S),\PSL_2(\C))=\mathbb{X}_{\mathrm{A}^2}(\pi_1(S),\PSL_2(\C))$ is a proper subset of $W(S)$, and

\item
$W(S)$ contains representations that are neither discrete, nor faithful.

\end{enumerate}
\end{Theorem}

\subsection{3-manifold groups}
\label{3manifoldgroups}

Suppose that $M$ is a compact, hyperbolizable 3-manifold whose boundary is nonempty with no toroidal components and that $\pi_1(M)$ is one-ended and not a surface group. Sullivan \cite{sullivan} showed that the set $\mathbb{X}_{\mathrm{A}}(\pi_1(M),\PSL_2(\C))$ of convex cocompact representations is the interior of the set $AH(M)$ of discrete, faithful representations of $\pi_1(M)$ into $\PSL_2(\C)$. J\o rgensen \cite{jorgensen} showed that $AH(M)$ is a closed subset of $\mathbb{X}(\pi_1(M),\PSL_2(\C))$.

Canary and Storm \cite{Canary-Storm} described a domain of discontinuity for the action of $\Out(\Gamma)$ on $\mathbb{X}(\pi_1(M),\PSL_2(\C))$ that is strictly larger than $\mathbb{X}_{\mathrm{A}}(\pi_1(M),\PSL_2(\C))$ and contains representations that are not discrete and faithful. It seems likely that \hbox{$\mathbb{X}_{\mathrm{A}^2}(\pi_1(M),\PSL_2(\C))$} is contained in their domain of discontinuity. The following result is immediate from their analysis.

\begin{Theorem}
\label{CS}
Suppose that $M$ is a compact, hyperbolizable 3-manifold whose boundary is nonempty with no toroidal components and that $\pi_1(M)$ is one-ended and not a surface group. Then $\mathbb{X}_{\mathrm{A}}(\pi_1(M),\PSL_2(\C))$ is a proper subset of $\mathbb{X}_{\mathrm{A}^2}(\pi_1(M),\PSL_2(\C))$.

Moreover, the set
\[
\mathbb{X}_{\mathrm{A}^2}(\pi_1(M),\PSL_2(\C)) \smallsetminus \mathbb{X}_{\mathrm{A}}(\pi_1(M),\PSL_2(\C))
\]
of amalgam Anosov representations that are not Anosov, contains representations that are discrete and faithful and representations that are not discrete and faithful.
\end{Theorem}

\begin{proof}
Lemma \ref{AareAA} implies that $\mathbb{X}_{\mathrm{A}}(\pi_1(M),\PSL_2(\C))=\textrm{int}(AH(M))$ is a subset of $\mathbb{X}_{\mathrm{A}^2}(\pi_1(M),\PSL_2(\C))$.

Lemma 4.2 in Canary--Hersonsky \cite{canary-hersonsky} guarantees that $\partial AH(M)$ contains representations that are purely hyperbolic, i.e., the image of every nontrivial element of $\pi_1(M)$ is a hyperbolic element of $\PSL_2(\C)$. Let \hbox{$\rho\in \partial AH(M)$} be purely hyperbolic. Lemma 8.2 in Canary--Storm \cite{Canary-Storm} implies that the restriction of $\rho$ to any Fuchsian vertex group is Anosov, while Lemma 8.3 in \cite{Canary-Storm} guarantees that for every cyclic vertex group there is a registering subgroup $H$ of $\pi_1(M)$ so that $\rho|_H$ is Anosov. It follows that $\rho$ is amalgam Anosov. Since $\rho\in\partial AH(M)$ it is not Anosov, so 
\[
\rho\in \mathbb{X}_{\mathrm{A}^2}(\pi_1(M),\PSL_2(\C)) \smallsetminus \mathbb{X}_{\mathrm{A}}(\pi_1(M),\PSL_2(\C)).
\]
In particular, $\mathbb{X}_{\mathrm{A}}(\pi_1(M),\PSL_2(\C))$ is a proper subset of $\mathbb{X}_{\mathrm{A}^2}(\pi_1(M),\PSL_2(\C))$.

Since $\rho$ is a smooth point of $\mathbb{X}(\pi_1(M),\PSL_2(\C))$ (see Kapovich \cite[Thm.\ 8.44]{kapovich}), $\rho\in \partial AH(M)\cap \mathbb{X}_{\mathrm{A^2}}(\pi_1(M),\PSL_2(\C))$ and $\mathbb{X}_{\mathrm{A}^2}(\pi_1(M),\PSL_2(\C))$ is open, there exist representations in \hbox{$\mathbb{X}_{\mathrm{A}^2}(\pi_1(M),\PSL_2(\C)) \smallsetminus \mathbb{X}_{\mathrm{A}}(\pi_1(M),\PSL_2(\C))$} that do not lie in $AH(M)$, and hence are not discrete and faithful.
\end{proof}

\noindent
{\bf Remark:} If  the JSJ-splitting of $\pi_1(M)$ has no rigid vertices, e.g., if $M$ is a book of $I$-bundles, 
then the above proof shows that any purely hyperbolic representation in $\partial AH(M)$ is strongly amalgam
Anosov, so $\mathbb X_{\mathrm{A}}(\pi_1(M),\PSL_2(\C))$ is a proper subset of $\mathbb{X}_{\mathrm{SA}}(\pi_1(M),\PSL_2(\C))$.

\medskip

If $\tau : \PSL_2(\C) \to \SL_{15}(\R)$ is the Pl\"ucker representation, then Theorem \ref{CS} and Lemma \ref{Pluckeramalgam} immediately apply to establish the following corollary. (One may easily compute, by examining the proof, that one can take
$d(\PSL_2(\C))=15$ in Proposition \ref{Plucker}.)

\begin{Corollary}
Suppose that $M$ is a compact, hyperbolizable 3-manifold whose boundary is nonempty and has no toroidal components and that $\pi_1(M)$ is one-ended and not a surface group. Then
\[
\mathbb{X}_{\mathrm{A}}(\pi_1(M),\SL_{15}(\R)) \subset \mathbb{X}_{\mathrm{A}^2}(\pi_1(M), \SL_{15}(\R))
\]
is a proper subset. In particular, the set $\mathbb{X}_{\mathrm{A}^2}(\pi_1(M), \SL_{15}(\R))$ contains representations that are not discrete and faithful, hence not Anosov.
\end{Corollary}

\subsection{Rank 1 Lie groups}\label{rankone}

For every noncompact rank one Lie group $G$, one can construct a nonrigid hyperbolic group $\Gamma$ such that $\Xa$ is empty 
and $\Xsa$ has positive dimension.

\begin{Proposition}\label{thm:spn1examples}
If $G$ is a noncompact, connected, rank one, simple Lie group with finite center, then
there exists a one-ended, torsion-free hyperbolic group $\Gamma$ such that
\begin{enumerate}

\item 
$\Xa = \emptyset$, 

\item
$\Out(\Gamma)$ is infinite, and

\item 
$\Xsa$ has positive dimension.

\end{enumerate}
Moreover, there is a continuous family $\{\rho_z\}$ of distinct representations in $\Xsa$ each of which is either indiscrete or not faithful.
\end{Proposition}

\begin{proof}
We first suppose that $G$ is not locally isomorphic to $\PSL_2(\R)$. Let $\Lambda$ be a torsion-free cocompact lattice in $G$. Then $\Lambda$ is a rigid one-ended hyperbolic group. Our basic construction is to double $\Lambda$. Let $C = \langle \sigma \rangle$ be a maximal infinite cyclic subgroup of $\Lambda$ and let $\Gamma$ be the double of $\Lambda$ along $C$. By definition, $\Gamma$ is then a graph of groups with two rigid vertex groups $\Lambda_1$ and $ \Lambda_2$ on either side of a single infinite cyclic vertex group. By the Bestvina--Feighn Combination Theorem (see the first corollary on page 100 of \cite{Bestvina-Feighn}), $\Gamma$ is again a torsion-free hyperbolic group. The JSJ splitting for $\Gamma$ has 3 vertices, one cyclic vertex between the two rigid vertices with group $\Lambda_i \cong \Lambda$ ($i = 1, 2$).
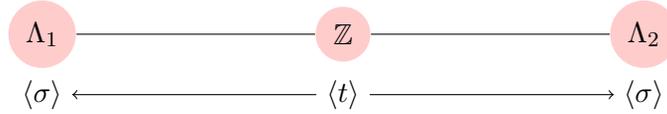
\begin{center}
\begin{tikzpicture}
[every node/.style={circle, fill=red!20}]
\node (n1) at (1,1) {$\Lambda_1$};
\node (n2) at (5,1) {$\Z$};
\node (n3) at (9,1) {$\Lambda_2$};

\foreach \from/\to in {n1/n2, n2/n3} \draw (\from) -- (\to);

\end{tikzpicture}
\begin{tikzpicture}
\node (n1) at (1, 1) {$\langle \sigma \rangle$};
\node (n2) at (5, 1) {$\langle t \rangle$};
\node (n3) at (9, 1) {$\langle \sigma \rangle$};

\foreach \from/\to in {n1/n2} \draw[<-] (\from) -- (\to);
\foreach \from/\to in {n2/n3} \draw[->] (\from) -- (\to);

\end{tikzpicture}
\end{center}

We first observe that no representation in $\X$ is discrete and faithful. Since $\Lambda$ is a cocompact lattice, $H_n(\Lambda, \Z) \cong \Z$ where $n$ is the dimension of the symmetric space $X$ associated with $G$. If $\rho \in \X$ is discrete and faithful, then, since
\[
H_n(\rho(\Lambda_1), \Z) \cong H_n(\Lambda, \Z) \cong \Z,
\]
we see that $X/\rho(\Lambda_1)$ is compact. However, this is impossible since $\rho(\Gamma)$ is discrete and 
$\rho(\Lambda_1)$ has infinite index in $\rho(\Gamma)$. Notice that this implies, in particular, that $\Xa$ is empty, since every Anosov representation of a torsion-free hyperbolic group is discrete and faithful.

We next construct a family of amalgam Anosov representations of $\Gamma$. Let $Z$ be the centralizer of $C \subset \Lambda$ in $G$. Then $Z$ contains the Zariski closure of the diagonalizable group $C$, and thus has positive dimension. Given $z\in Z$, we define $\rho_z\in\X$ to be the identity on $\Lambda_1$ and to take every $\gamma\in\Lambda_2$ to $z \gamma z^{-1}$. The restriction of $\rho_z$ to each $\Lambda_i$ is a discrete, faithful representation with image a lattice. Thus $\rho_z|_{\Lambda_i}$ is convex cocompact and hence Anosov. (Recall that a discrete, faithful representation into a rank 1 Lie group is Anosov if and only if it is convex cocompact; see Guichard--Wienhard \cite[Thm.\ 5.15]{Guichard-Wienhard}.) Therefore, $\rho_z$ is strongly amalgam Anosov for all $z\in Z$. Since the group $Z$ has positive dimension, this produces a positive dimensional subset of
$\Xsa$.

If $G$ is locally isomorphic to $\PSL_2(\R)$, we must choose $C$ carefully to carry out the above argument.
Let $\Lambda$ be a torsion-free, cocompact lattice in $G$, so
$S= \mathbb{H}^2 / \Lambda$ is a closed orientable surface, and let $C$ be a maximal infinite cyclic subgroup of $\Lambda$.
We again form  $\Gamma$ by doubling $\Lambda$ along $C$. Then $\Gamma$ has a $\Z$-splitting connecting two vertex groups, 
each isomorphic to $\Lambda$, and one cyclic vertex group associated with $C$. 
If a generator of $C$ is represented by
a simple closed curve $\alpha$ on $S$, then the noncyclic vertex groups of this splitting will not be rigid (as one will obtain a further
$\Z$-splitting associated with a simple closed curve disjoint from $\alpha$).
However, if a generator of  $C$ is represented by a filling curve $\alpha$ on $S$, i.e., a curve that essentially intersects every homotopically nontrivial simple closed curve and we again form $\Gamma$ by doubling $\Lambda$ along $C$, then the noncyclic vertex groups will be
rigid, since any $\mathbb{Z}$-splitting of $\Lambda$ arises from a simple closed curve on $S$. Since this simple closed
curve must intersect $\alpha$, the edge group associated to $C$ cannot lie in a vertex group of the splitting.
With this choice of $\Lambda$ and $C$, our splitting is a JSJ splitting and we can complete the proof as above.
\end{proof}

If $\tau_G:G\to \SL_{d(G)}(\R)$ is the Pl\"ucker representation given by Proposition \ref{Plucker}, then Lemma \ref{Pluckeramalgam} implies that each $\tau_G\circ \rho_z$ is amalgam Anosov but not discrete and faithful, hence not Anosov. Therefore we obtain the following immediate corollary.

\begin{Corollary}
There exist infinitely many distinct values of $d$ so that there exists a torsion-free one-ended hyperbolic group $\Gamma_d$ 
such that $\Out(\Gamma_d)$ is infinite and
\[
\mathbb{X}_{\mathrm{A}}(\Gamma_d,\SL_d(\R))\ne \mathbb{X}_{\mathrm{A}^2}(\Gamma_d,\SL_d(\R)).
\]
\end{Corollary}

\subsection{Amalgam Anosov representations that are not strongly amalgam Anosov}
\label{aanotsa}

In this section we modify the construction in Theorem \ref{thm:spn1examples} to  build amalgam Anosov representations of nonrigid groups that are not strongly amalgam Anosov. 

\begin{Proposition}
\label{notSA}
There exist infinitely many one-ended, torsion-free hyperbolic groups $\Gamma$ such that $\Out(\Gamma)$ is infinite and
there exists an amalgam Anosov representation $\rho:\Gamma\to\PSL (2,\C)$ that is not strongly amalgam Anosov.\end{Proposition}

\begin{proof}
Let $M$ be a compact, acylindrical, hyperbolizable 3-manifold and let $\sigma:\pi_1(M)\to \PSL_2(\C)$ be a discrete, faithful, purely hyperbolic representation that is not Anosov (see Theorem \ref{CS}). Let $C$ be a maximal cyclic subgroup of $\pi_1(M)$ and let $\Gamma$ be the double of $\pi_1(M)$ along $C$. Again $\Gamma$ is hyperbolic and its JSJ splitting has two rigid vertices, each isomorphic to $\pi_1(M)$, and a single cyclic vertex group associated to $C$. We construct a representation $\rho$ by letting $\rho$ agree with $\sigma$ on each copy of $\pi_1(M)$ in $\Gamma$. Then, by construction, $\rho$ is not strongly amalgam Anosov.

On the other hand, we can choose elements $g_1$ and $g_2$ of $\pi_1(M)$ so that the fixed points of $\rho(g_1)$, $\rho(g_2)$ and $\rho(c)$ in $\mathbb{CP}^1=\partial_\infty \mathbb{H}^3$ are all disjoint, where $c$ is a generator of $C$. Therefore, again by a ping-pong argument, there exists $N$ so that the group generated by $\sigma(g_1^n)$, $\sigma(g_2^n)$ and $\sigma(c^n)$ is a Schottky group for $n\ge N$. In particular it is a convex cocompact free group of rank 3. Let $\gamma_1\in\Gamma$ be the copy of $g_1$ in the first copy of $\pi_1(M)$ and let $\gamma_2\in\Gamma$ be the copy of $g_2$ in the second copy of $\pi_1(M)$. Then, if $n\ge N$, the group generated by 
$\rho(\gamma_1^n)$, $\rho(\gamma_2^n)$ and $\rho(c^n)$ is free of rank 3 and convex cocompact, so the subgroup $H$ of $\Gamma$ generated by $\gamma_1^n$, $\gamma_2^n$ and $c^n$ registers $\Gamma_v$, where $v$ is the cyclic vertex of the JSJ decomposition of $\Gamma$, and $\rho|_H$ is Anosov. Therefore, $\rho$ is amalgam Anosov.

Since there are infinitely many choices for $M$, there are infinitely many choices for $\Gamma$.
\end{proof}

\begin{Remark}
In Propositions \ref{thm:spn1examples} and \ref{notSA}, one can construct examples where $\Out(\Gamma)$ is not virtually abelian. For example, if $G=\PSL_2(\C)$, one may choose $\Gamma$ to be the amalgamation of either a lattice $\Lambda$ or the fundamental group of a compact acylindrical 3-manifold $M$ with $\pi_1(F)$, where $F$ is a compact surface of genus at least one with connected boundary and one identifies a maximal cyclic subgroup of $\Lambda$ or $\pi_1(M)$ with a maximal cyclic group in the conjugacy class of $\pi_1(\partial F)$.
\end{Remark}

\subsection{$\Xaa$ need not be a maximal domain of discontinuity.}
\label{notmax}

We now exhibit examples where $\Xaa$ is not a maximal domain of discontinuity for the action of $\Out(\Gamma)$ on $\X$.

\begin{Proposition}
\label{Xaanotmaximal}
There exists a compact, hyperbolizable 3-manifold $M$
such that
\begin{enumerate}
\item
$\pi_1(M)$ is one-ended, torsion-free and hyperbolic,
\item 
$\Out(\pi_1(M))$ is not virtually abelian, and
\item 
there exists an open, $\Out(\Gamma)$-invariant subset $Y(M)$ of $\mathbb{X}(\pi_1(M),\PSL_2(\C))$ 
so that $\mathbb{X}_{\mathrm{A}^2}(\pi_1(M),\PSL_2(\C))$ is a proper subset of $Y(M)$ and
$\Out(\pi_1(M))$ acts properly discontinuously on $Y(M)$.
\end{enumerate}
\end{Proposition}

\begin{proof}
Let $M_0$ be a compact, acylindrical, hyperbolizable 3-manifold without toroidal boundary components such that there exists an epimorphism $\alpha : \pi_1(M_0) \to F_2$, where $F_2$ is the free group on 2 elements. Furthermore, suppose that there exists a simple closed curve $c$ on $\partial M_0$ such that $\alpha(c)$ is a generator of $F_2$. Let $S$ be a compact, orientable 2-holed surface of genus 3 with boundary components $\partial S^1$ and $\partial S^2$.

Let $M_0^1$ and $M^2_0$ be two copies of $M_0$ containing copies $c^1$ and $c^2$ of $c$. We form $M$ from $S \times [0,1]$, $M_0^1$, and $M_0^2$ by identifying $\partial S^i\times [0,1]$ with a collar neighborhood $C^i$ of $c^i$ in $\partial M_0^i$ for $i=1,2$.
Results of Bestvina--Feighn \cite{Bestvina-Feighn} again imply that $\pi_1(M)$ is hyperbolic. Since $S \times [0,1]$, $M_0^1$, and $M_0^2$ are irreducible, and they
are attached to one another along incompressible annuli to form $M$, $M$ is also irreducible.
Thurston's Geometrization Theorem (see Morgan \cite{morgan}) then implies that $M$ is hyperbolizable.
Moreover, there is a graph of groups decomposition of $\pi_1(M)$ which has 5 vertices, two rigid vertices identified with $\pi_1(M_0^1)$ and $\pi_1(M_0^2)$, two cyclic vertices, identified with $\pi_1(C^1)$ and $\pi_1(C^2)$, and one Fuchsian vertex, identified with $\pi_1(S)$. Therefore, $\pi_1(M)$ is torsion-free, freely indecomposable, and
not cyclic, hence $\pi_1(M)$ is one-ended and the given graph of groups decomposition is the JSJ splitting:
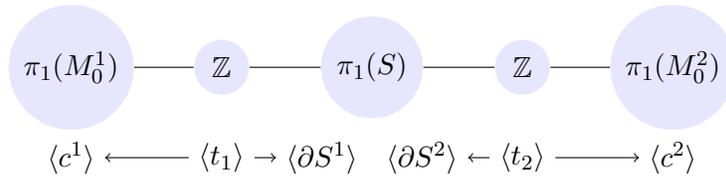
\begin{center}
\begin{tikzpicture}
[every node/.style={circle, fill=blue!10}]
\node (n1) at (1,1) {$\pi_1(M_0^1)$};
\node (n2) at (3,1) {$\Z$};
\node (n3) at (5,1) {$\pi_1(S)$};
\node (n4) at (7,1) {$\Z$};
\node (n5) at (9,1) {$\pi_1(M_0^2)$};

\foreach \from/\to in {n1/n2, n2/n3, n4/n3, n4/n5} \draw (\from) -- (\to);

\end{tikzpicture}
\begin{tikzpicture}
\node (n1) at (1, 1) {$\langle c^1 \rangle$};
\node (n2) at (3, 1) {$\langle t_1 \rangle$};
\node (n3) at (5, 1) {$\langle \partial S^1 \rangle \quad \langle \partial S^2 \rangle$};
\node (n4) at (7,1) {$\langle t_2 \rangle$};
\node (n5) at (9,1) {$\langle c^2 \rangle$};

\foreach \from/\to in {n1/n2} \draw[<-] (\from) -- (\to);
\foreach \from/\to in {n2/n3} \draw[->] (\from) -- (\to);
\foreach \from/\to in {n4/n3} \draw[->] (\from) -- (\to);
\foreach \from/\to in {n4/n5} \draw[->] (\from) -- (\to);

\end{tikzpicture}
\end{center}

Let $W(M)$ be the domain of discontinuity constructed by Canary and Storm \cite{Canary-Storm} for the action of $\Out(\pi_1(M))$ on $\mathbb{X}(\pi_1(M),\PSL_2(\C))$. Recall that $\rho\in W(M)$ if and only if
\begin{enumerate}
\item
the restriction of $\rho$ to the free group $\pi_1(S)$ is primitive-stable, and
\item
for $i = 1, 2$, there exists a subgroup $H^i$ of $\pi_1(M)$ freely generated by $c^i = \partial S^i$, an element of $\pi_1(M_0^i)$, 
and an element of $\pi_1(S)$ such that the restriction of $\rho$ to $H^i$ is primitive-stable.
\end{enumerate}
We refer the reader to Minsky \cite{Minsky} for the definition of a primitive-stable representation of a free group $F_n$, but we recall that the set of primitive-stable representations is a domain of discontinuity for the action of $\Out(F_n)$ on $\mathbb{X}(F_n,\PSL_2(\C))$ that contains $\mathbb{X}_{\mathrm{A}}(F_n,\PSL_2(\C))$ as a proper subset.

Let
\[
Y(M)=W(M) \cup \mathbb{X}_{\mathrm{A}^2}(\pi_1(M),\PSL_2(\C)).
\]
Then $Y(M)$ is an $\Out(\pi_1(M))$-invariant open subset of $\mathbb{X}(\pi_1(M),\PSL_2(\C))$ on which $\Out(\pi_1(M))$ acts properly discontinuously. It only remains to prove that
\[
W(M) \smallsetminus \mathbb{X}_{\mathrm{A}^2}(\pi_1(M),\PSL_2(\C))
\]
is nonempty.

Let $\rho_S : \pi_1(S) \to \PSL_2(\C)$ be a primitive-stable representation that is not convex cocompact, hence not Anosov. For $i=1,2$, 
let $J^i$ be a proper free factor of rank 3 of $\pi_1(S)$ so that $\partial S^i$ is an element of a minimal generating set $\{\partial S^i, a_2^i, a_3^i\}$ for $J^i$.
(For example, one may choose $J^i$ to be the fundamental group of an essential twice-punctured torus subsurface of $S$ with $\partial S^i$ as
one boundary component.)
Lemma 3.2 in Minsky \cite{Minsky} then implies that the restriction of $\rho_S$ to $J^i$ is convex cocompact. Let $g$ be an element of $\pi_1(M_0)$ such that $\alpha(g)$ and $\alpha(c)$ generate $F_2$ and let $g^i$ be the copy of $g$ in $M_0^i$. Let $\alpha^i:\pi_1(M_0^i)\to F_2$ be the result of precomposing $\alpha$ with the identification of $\pi_1(M_0^i)$ with $\pi_1(M_0)$ and define $\nu^i : F_2 \to \PSL_2(\C)$ by $\nu^i(\alpha^i(c^i)) = \rho_S(\partial S^i)$ and $\nu^i(\alpha^i(g^i))=\rho_S(a_3^i)$.

We define $\rho : \pi_1(M) \to \PSL_2(\C)$ to agree with $\rho_S$ on $\pi_1(S)$ and agree with $\nu^i\circ\alpha^i$ on
$\pi_1(M_0^i)$ for each $i$. For each $i$, let $H^i$ be the subgroup of $\pi_1(M)$ generated by $c^i$, $a_2^i$ and $g^i$. Then $\rho(H^i)=\rho_S(J^i)$ is a convex cocompact free group on 3 generators, which implies that $H^i$ is a free group on 3 generators and the restriction of $\rho$ to $H^i$ is convex cocompact. It follows that $\rho \in W(M)$. However, since $\rho_S$ is not Anosov, $\rho$ is not amalgam Anosov. This completes the proof.
\end{proof}

\medskip\noindent
{\bf Remarks:}
(1) If one has a domain of discontinuity $P(F_n,G)$ for the action of
 $\Out(F_n)$ on $\X$ that is strictly larger than $\mathbb{X}_{\mathrm{A}}(F_n,G)$ for all $n\ge 2$, and the JSJ splitting of $\Gamma$ is nontrivial
and contains a Fuchsian vertex group, then one may generalize the construction of Canary--Storm \cite{Canary-Storm} by
considering the set $Q(\Gamma,G)\subset \X$ of representations so that (1) if $\Gamma_v$ is a Fuchsian vertex group, then 
$\rho|_{\Gamma_v}\in P(\Gamma_v,G)$, and (2) if $\Gamma_v$ is a cyclic vertex group, then there exists a group $H$
registering $\Gamma_v$ so that $\rho|_H\in P(\Gamma_v,G)$. The set $Q(\Gamma,G)$ will be a domain of discontinuity for the action of
$\Out(\Gamma)$ on $\X$ that contains $\Xaa$. It seems likely that it would often be the case that $Q(\Gamma,G)$ strictly contains $\Xaa$.
 
(2) It is an open question whether or not $W(M)$ is a maximal domain of discontinuity for the action of $\Out(\pi_1(M))$ on $\mathbb{X}(\pi_1(M),\PSL_2(\C))$. This is related to the question of whether or not the set of primitive-stable representations is a maximal domain of discontinuity for the action of $\Out(F_n)$ on $\mathbb{X}(F_n,\PSL_2(\C))$.

\newpage

\appendix

\section{Appendix: Projective Anosov Schottky groups and strongly amalgam Anosov representations}\label{Appendix}

\medskip

\begin{center}
{RICHARD D.\ CANARY, MICHELLE LEE, ANDR\'ES SAMBARINO, AND MATTHEW STOVER}
\end{center}

\medskip

The purpose of this appendix is to prove that strongly amalgam Anosov representations are amalgam Anosov. We recall that if $\Gamma$ is
a one-ended, torsion-free hyperbolic group and $G$ is a semisimple Lie group with finite center, then $\rho\in\X$ is \emph{strongly amalgam Anosov} if there exists a pair 
$P^\pm$ of proper opposite parabolic subgroups so that $\rho|_{\Gamma_v}$ is $(P^+,P^-)$-Anosov for every Fuchsian or rigid vertex group $\Gamma_v$ in the JSJ decomposition of $\Gamma$.

\begin{Theorem}
\label{main-appendix}
Suppose that $\Gamma$ is a one-ended, torsion-free word hyperbolic group and $G$ is a semisimple Lie group with finite center. If $\rho \in \X$ is strongly amalgam Anosov, then $\rho$ is amalgam Anosov.
\end{Theorem}

The proof is based on the fact, see Theorem \ref{thm:ConvexSchottky} below, that if one is given a well-positioned collection of biproximal elements of $\SL_d(\R)$, then one may pass to large enough powers that the group they generate is a free, projective Anosov group, i.e., a projective Anosov Schottky group. The existence of a result of this form has been well-known among experts. For example, in the typical situation where the Schottky group produced by Benoist \cite[Sec. 6]{BenoistAnnals}
is irreducible, it follows immediately that it is projective Anosov from results of Quint \cite[Prop. 3.3]{QuintDynam} and
Guichard-Wienhard \cite[Prop. 4.10]{Guichard-Wienhard}.
Kapovich, Leeb, and Porti \cite[Theorem 7.40]{KLP} previously established a
generalization of Theorem \ref{thm:ConvexSchottky} to the setting of general semisimple Lie groups.
Our proof of Theorem \ref{thm:ConvexSchottky} relies heavily on results and techniques of Benoist \cite{BenoistAnnals,Benoist} and Quint \cite{QuintDynam} as well as a criterion for a representation to be projective Anosov due to Gu\'eritaud--Guichard--Kassel--Wienhard \cite{GGKW}.

Recall that a matrix $\gamma \in \SL_d(\R)$ is \emph{proximal} if its eigenvalue of maximal modulus is real and of multiplicity one. It is \emph{biproximal} if both $\gamma$ and $\gamma^{-1}$ are proximal. If $\gamma$ is proximal, then it has a unique attracting fixed point $\gamma^+ \in \RPd$, which is the eigenline associated to the eigenvalue of maximal modulus. Let $\gamma^- \in \RPddual$ denote the repelling hyperplane of $\gamma$. (We regularly use the identification of $\RPddual$ with the Grassmannian of hyperplanes in $\R^d$.)
If $\rho:\Gamma\to \SL_d(\R)$ is projective Anosov with associated limit maps $\xi:\partial\Gamma\to \RPd$ and $\theta:\partial\Gamma\to\RPddual$
and $c\in\partial_\infty\Gamma$ is an attracting fixed point of an element $\gamma\in\Gamma$, then $\rho(\gamma)$ is biproximal and 
\[
\xi(c)=\rho(\gamma)^+\quad\textrm{and}\quad \theta (c)=(\rho(\gamma)^{-1})^-
\]
(see Lemma \ 3.1 and Corollary \ 3.2 in Guichard--Wienhard \cite{Guichard-Wienhard}).

A set $S$ of matrices is \emph{symmetric} when $\gamma \in S$ if and only if $\gamma^{-1} \in S$. A symmetric set $S = \{\gamma_1,\ldots,\gamma_{2r}\}$ of biproximal matrices in $\SL_d(\R)$ is \emph{well-positioned} if,
\begin{enumerate}
\item
$\gamma_i^+ \neq \gamma_j^+$ if $i\ne j$, and
\item
$\gamma_i^+ $ is not contained in the hyperplane $\gamma_j^-$ whenever $\gamma_j \neq \gamma_i^{-1}$.
\end{enumerate}
Notice that conditions (1) and (2) immediately imply that $\gamma_i^- \neq \gamma_j^-$ if $i\ne j$.

The key technical result needed for the proof of Theorem \ref{main-appendix} is the following.

\begin{Theorem}
\label{thm:ConvexSchottky}
{\rm (Tits \cite{tits-ping}, Benoist \cite{BenoistAnnals}, Quint \cite{QuintDynam}, Kapovich-Leeb-Porti \cite{KLP})}
If 
\hbox{$S = \{\gamma_1, \dots, \gamma_{2 r}\}$} is a well-positioned symmetric set of biproximal matrices in $\SL_d(\R)$,
then there exists $N > 0$ such that if $n\ge N$, then the subgroup of $\SL_d(\R)$ generated by 
\[
S^n = \{\gamma_1^n,\ldots,\gamma^n_{2r}\}
\]
is a free group of rank $r$ and the induced representation \hbox{$\rho : F_r \to \SL_d(\R)$} is projective Anosov.
\end{Theorem}

We now prove Theorem \ref{main-appendix} assuming Theorem \ref{thm:ConvexSchottky}, and return to the proof of Theorem \ref{thm:ConvexSchottky} afterwards.

\begin{proof}[Proof of Theorem \ref{main-appendix}]
We first suppose that $\rho:\Gamma\to \SL_d(\R)$ is a representation such that the restriction of $\rho$ to any Fuchsian or rigid vertex group of the JSJ decomposition of $\Gamma$ is projective Anosov. If $w$ is a Fuchsian or rigid vertex of the JSJ decomposition of $\Gamma$, let
\[
\xi_w:\partial_\infty\Gamma_w\to \RPd\quad\textrm{and}\quad\theta_w:\partial_\infty\Gamma_w\to\RPddual
\]
be the limit maps for $\rho|_{\Gamma_w}$.

Let $v$ be a cyclic vertex and let $\{w_1,\ldots,w_r\}$ be the vertices adjacent to $v$ and $a$ be the generator for $\Gamma_v$. Since $\Gamma_{w_1}$ is nonelementary, pairs of fixed points of elements of $\Gamma_{w_1}$ are dense in $\partial\Gamma_{w_1} \times \partial\Gamma_{w_1}$. Thus, since $\xi_{w_1}$ is injective and $\theta_{w_1}$ is transverse to $\xi_{w_1}$, we can find $g_1\in \Gamma_{w_1}$ so that $\rho(g_1)^+$ is distinct from $\rho(a)^+$ and not contained in $\rho(a)^-$, and $\rho(g_1)^-$ is distinct from $\rho(a)^-$ and does not contain $\rho(a)^+$. Therefore, $\{\rho(a)^{\pm 1},\rho(g_1)^{\pm 1}\}$ is a well-positioned symmetric set of biproximal matrices.

Similarly, we iteratively choose $g_i\in\Gamma_{w_i}$ so that
\[
\{\rho(a)^{\pm 1},\rho(g_1)^{\pm 1},\ldots,\rho(g_i)^{\pm 1}\}
\]
is a well-positioned symmetric set of biproximal matrices. Let 
\[
S=\{\rho(a)^{\pm 1},\rho(g_1)^{\pm 1},\ldots,\rho(g_r)^{\pm 1}\}
\]
be the final result of this process. Theorem \ref{thm:ConvexSchottky} implies that there is a positive integer $N$ such that if $n\ge N$, then
\[
\Lambda_n = \langle S^n \rangle
\]
is a free group of rank $r + 1$ and the restriction of $\rho$ to $\Lambda_n$ is projective Anosov. Therefore, for $n \ge N$, 
$H_n=\langle a^n,g_1^n,\ldots,g_r^n\rangle$ is a registering subgroup for $\Gamma_v$ and the restriction of $\rho$ to $H_n$ is projective Anosov. Therefore, for every cyclic vertex group there exists a registering subgroup $H$ such that the restriction of $\rho$ to $H$ is projective Anosov.
Since the restriction of $\rho$ to every Fuchsian vertex group is projective Anosov by assumption, it follow that
$\rho$ is amalgam Anosov.

Now suppose that $\rho:\Gamma\to G$ is strongly amalgam Anosov and $P^\pm$ is a pair of proper opposite parabolic subgroups so that the restriction of $\rho$ to each Fuchsian or rigid vertex group is $(P^+,P^-)$-Anosov. Let \hbox{$\tau:G\to\SL_d(\R)$} be the irreducible representation provided by Proposition \ref{Plucker}. Then the restriction of $\tau\circ\rho$ to every rigid and Fuchsian vertex group is projective Anosov. It follows from the above analysis that for every cyclic vertex group there exists a registering subgroup $H$ such that the restriction of $\tau\circ\rho$ to $H$ is projective Anosov, which implies that the restriction of $\rho$ to $H$ is $(P^+,P^-)$-Anosov. Therefore, $\rho$ is amalgam Anosov.
\end{proof}

We now turn to the proof of Theorem \ref{thm:ConvexSchottky}. We first recall
Gu\'eritaud, Guichard, Kassel, and Wienhard's \cite{GGKW} criterion for a representation to be projective Anosov.
(We have been informed that one can derive an analogue of this criterion, see Theorem \ref{GGKW criterion} below, 
from the work in \cite{KLP}.)

Choose the maximal torus $\mathfrak{a}$ for the Lie algebra $\mathfrak{sl}_d(\R)$ of $\SL_d(\R)$ for which $\exp(\mathfrak{a})$ is the group of diagonal matrices in $\SL_d(\R)$. We can then choose a positive Weyl chamber $\mathfrak{a}^+$ to be the interior of $\bar{\mathfrak{a}}^+$, where $\exp(\overline{\mathfrak{a}}^+)$ is the set of positive diagonal matrices with entries in descending order (from left to right). If $\alpha_1$ is the root of highest weight for $\mathfrak{a}^+$, then
\[
\alpha_1(a_1,a_2,\ldots,a_d)=a_1-a_2.
\]

Let $K \exp(\overline{\mathfrak{a}}^+) K$ be the Cartan decomposition of $\SL_d(\R)$ with respect to $\mathfrak{a}^+$ and a maximal compact subgroup $K=\mathrm{SO}(d)$,
\[
\mu : G \to \bar{\mathfrak{a}}^+
\]
be the associated Cartan projection (see \cite{BenoistAnnals, QuintDynam, Guichard-Wienhard}), and $\mu_i(g)$ denote the $i^{\rm th}$ entry of $\mu(g)$. Fix a norm $\| \ \|$ on $\mathbb R^d$ arising from a $K$-invariant inner product and 
let $\|\gamma\|$ be the operator norm of a linear transformation $\gamma$
with respect to $\|\ \|$.

Notice that
\[
\mu_1(g)=\log(||g||).
\]
Similarly, let 
\[
\lambda : G \to \overline{\mathfrak{a}}^+
\]
be the Jordan projection (see \cite{Benoist, QuintDynam, Guichard-Wienhard}), and let $ \lambda_i(g)$ denote the $i^{\rm th}$ entry of $\lambda(g)$. If $g$ is proximal, then $\Lambda_1(g)=e^{\lambda_1(g)}$ is the modulus of the eigenvalue of maximal modulus.

If $\Gamma$ is a hyperbolic group, we recall that maps
\[
\xi : \partial_\infty \Gamma \to \RPd \quad \textrm{and}\quad\theta: \partial_\infty \Gamma \to \RPddual
\]
are \emph{transverse} if $\xi(x)$ and $\theta(y)$ span $\R^d$ for all distinct elements $x\ne y$ in $\partial_\infty \Gamma$. Moreover, $\xi$ and $\theta$ are said to be {\em dynamics preserving} if, whenever $c\in\partial_\infty\Gamma$ is an attracting fixed point of $\gamma\in\Gamma$, $\xi(c)$ is an attracting eigenline of $\gamma$ and $\theta(c)$ is an attracting hyperplane of $\gamma$. 

\begin{Theorem}
{\rm (Gu\'eritaud--Guichard--Kassel--Wienhard \cite[Thm.\ 1.3]{GGKW})}
\label{GGKW criterion}
Suppose that $r\ge 2$ and $\rho:F_r\to \SL_d(\R)$ is a representation so that there exist transverse, dynamics preserving, $\rho$-equivariant, injective, continuous maps
\[
\xi : \partial_\infty F_r \to \RPd \quad \textrm{and}\quad\theta: \partial_\infty F_r \to \RPddual.
\]
If
\[
\lim_{n \to \infty} \alpha_1(\mu(x_0\ldots x_n)) = +\infty
\]
for every infinite reduced word $(x_n)$ in the generators of $F_r$, then $\rho$ is projective Anosov.
\end{Theorem}

We will show that Benoist's construction of Schottky groups in \cite{BenoistAnnals} can be slightly modified
so that the Schottky groups produced
satisfy the assumptions of Theorem \ref{GGKW criterion}.
Benoist \cite{BenoistAnnals} defines the distance between two lines $X, Y \in \RPd$ to be 
\[
d(X,Y)=\min\{\|x-y\|\ |\ x\in X, y\in Y, \|x\|=1,\|y\|=1\}.
\]
Given a subset $Z$ of $\RPd$ and $\epsilon>0$, define:
\begin{align*}
b(Z, \epsilon) &= \{X \in \RPd\ :\ d(X, Z) \le \epsilon\} \\
B(Z, \epsilon) &= \{X \in \RPd\ :\ d(X, Z) \ge \epsilon\}
\end{align*}

Following Benoist \cite{BenoistAnnals}, we say that a proximal matrix $\gamma$ is $\epsilon$-\emph{proximal} if 
\begin{enumerate}

\item
$d(\gamma^+, \gamma^-) \ge 2 \epsilon$, 

\item
$\gamma$ is $\epsilon$-Lipschitz on $B(\gamma^-, \epsilon)$, and

\item
$\gamma(B(\gamma^-, \epsilon)) \subset b(\gamma^+, \epsilon)$.
\end{enumerate}
Benoist \cite{BenoistAnnals} further says that a symmetric set $S = \{\gamma_1, \dots, \gamma_{2 r}\}$ of \hbox{$\epsilon$-proximal} matrices is $\epsilon$-\emph{Schottky} if $d(\gamma_i^+, (\gamma_j^{-1})^-) \ge 6 \epsilon$ for all $i \neq j$. Notice that this implies that $b(\gamma_i^+, \epsilon) \subset B(\gamma_j^-, \epsilon)$ whenever $\gamma_j \neq \gamma_i^{-1}$.

For an element $\gamma \in \SL_d(\R)$, let $\gamma^*$ be the element of $\SL((\R^d)^*)$ defined by
\[
\gamma^* (\theta) = \theta \circ \gamma^{-1}
\]
for all $\theta\in(\R^d)^*$. Recall that if $\gamma$ is a biproximal element of $\SL_d(\R)$, then
\[
(\gamma^*)^+ = ((\gamma^{-1})^-)^* \quad \textrm{and}\quad(\gamma^*)^- = ((\gamma^{-1})^+)^*.
\]
We will use the following additional properties of $\epsilon$-Schottky sets.

\begin{Lemma}\label{lem:ThanksFanny}
If $\epsilon>0$ and $S = \{\gamma_1, \dots, \gamma_{2 r}\}$ and $S^* = \{\gamma_1^*, \dots, \gamma_{2 r}^*\}$ are
$\epsilon$-Schottky sets of proximal matrices in $\SL_d(\R)$ and $\SL((\R^d)^*)$, then
\begin{enumerate}

\item
$V= \bigcap_{1 \le i \le 2 r} B(\gamma_j^-, \epsilon)$ is nonempty,
\item
$W = \bigcap_{1 \le i \le 2 r} B((\gamma_i^*)^-, \epsilon)\subset\RPddual$ is nonempty, and

\item
if $i \neq j$, $X \in b(\gamma_i^+, \epsilon)$, and $Y \in b((\gamma_j^*)^+, \epsilon)$, then $X$ and $Y$ span $\R^d$.

\end{enumerate}
\end{Lemma}

\begin{proof}
Since $S$ is symmetric and $\epsilon$-Schottky, if $v\in\RPd$ and $d(v,\gamma_i^+)<5\epsilon$, then
$d(v,\gamma_j^-)>\epsilon$ if $\gamma_j\ne\gamma_i^{-1}$. Therefore, $V$ contains any point such that
$d(v,\gamma_1^+) <5\epsilon$ and $d(v,(\gamma_1^{-1})^-)>\epsilon$,
so (1) holds. 
One can apply the same argument as in (1) to $S^*$ to verify (2).

Finally, to prove (3) we must show that if $X$ is a line in $b(\gamma_i^+, \epsilon)$ and $Y$ is a hyperplane
in $b((\gamma_j^*)^+, \epsilon)$, then $X$ is not contained in $Y$. However, any line 
contained in $Y$ lies within $\epsilon$ of a line in $(\gamma_j^{-1})^-$, which would imply that $d(\gamma_j^+,\gamma_j^-)<2\epsilon$,
which contradicts the fact that $S$ is $\epsilon$-Schottky.
\end{proof}

We now apply techniques of Benoist \cite{BenoistAnnals} and Quint \cite{QuintDynam} to produce limit maps for the Schottky groups generated by $\epsilon$-Schottky groups of matrices in $\SL_d(\R)$. These techniques are another instance of the ping-pong construction
which goes back to Klein \cite{klein-ping} in the Kleinian setting and Tits \cite{tits-ping} in the general linear setting.
The following proposition is our version of Proposition 3.3 in Quint \cite{QuintDynam}, which in turn
sharpens the discussion in Section 6.5 of Benoist \cite{BenoistAnnals}.

\begin{Proposition}\label{prop:LimitMaps}
{\rm (Quint \cite{QuintDynam})}
If $\epsilon>0$ and $S$ and $S^*$ are $\epsilon$-Schottky symmetric sets of proximal matrices in $\SL_d(\R)$ and $\SL((\R^d)^*)$,
then the subgroup 
$\Gamma$ of $\SL_d(\R)$ generated by $S$ is free of rank $r$. Furthermore, there exist transverse, dynamics preserving, $\Gamma$-equivariant, injective, continuous maps
\[
\xi : \partial_\infty \Gamma \to \RPd \quad \textrm{and} \quad \theta : \partial_\infty \Gamma \to \RPddual
\]
such that if $c\in\partial_\infty\Gamma$ is the attracting fixed point of an element $\gamma\in\Gamma$, then 
$\xi(c)$ is an attracting eigenline of $\gamma$ and $\theta(c)$ is an attracting hyperplane for $\gamma$.
\end{Proposition}

\begin{proof}
By Lemma \ref{lem:ThanksFanny}, 
\[
V= \bigcap_{1 \le i \le 2 r} B(\gamma_j^-, \epsilon)
\]
is nonempty. Choose $v\in V$, and let $x = x_0 x_1 \cdots x_n$ be any reduced word in the elements of $S$. Then, one may iteratively check that
\[
x(v) \in x_0 x_1 \cdots x_{i - 1} (b(x_i^+, \epsilon))\subset b(x_0,\epsilon).
\]
In particular, $x(v)$ does not lie in $V$, since it does not lie in $B((x_0^{-1})^-,\epsilon)$, so $x \ne 1$. Therefore, $\Gamma$ is free on $r$ generators. Each element $x_j\in S$ is \hbox{$\epsilon$-Lipschitz} on $b(x_j^+, \epsilon)$, so $x_0 \cdots x_{i - 1} (b(x_i^+, \epsilon))$ has diameter at most $\epsilon^i D$, where $D$ is the diameter of $\RPd$.

Recall that the boundary of a free group can be identified with the set of infinite reduced words in its generators.
If $x = (x_i)$ is an infinite reduced word in the elements of $S$, then it follows that $\{x_0x_1\cdots x_n(v)\}$ is a Cauchy sequence and we can define
\[
\xi(x) = \lim_{n \to \infty} x_0 \cdots x_n (v).
\]
If $w$ is any point in $V$, then $x_0\cdots x_n(w)$ also lies in $x_0 x_1 \cdots x_{i - 1} (b(x_i^+, \epsilon))$, which has diameter at most
$\epsilon^nD$. It follows that $\xi(x)=\lim_{n \to \infty} x_0 \cdots x_n (w)$. In particular, if $c\in\partial_\infty\Gamma$
is the attracting fixed point of $\gamma\in\Gamma$, then $\gamma^n(V)$ converges to $\xi(c)$ which implies that
$\xi(c)$ is an attracting eigenline of $\gamma$.

Given two infinite reduced words $x=(x_i)$ and $y=(y_i)$, let $r = r(x, y)$ be the first index at which they differ. Then
\[
\xi(x), \xi(y) \in x_0 \cdots x_{r - 2} (b(x_{r - 1}^+, \epsilon)),
\]
so $d(\xi(x), \xi(y)) \le \epsilon^{r - 1} D$, which implies that $\xi$ is continuous. Furthermore,
\[
\xi(x) \in x_0 \cdots x_{r - 1} (b(x_r^+, \epsilon)) \quad\textrm{and}\quad \xi(y) \in x_0 \cdots x_{r - 1}( b(y_r^+, \epsilon)),
\]
which are disjoint sets since $b(x_r^+, \epsilon)$ and $b(y_r^+,\epsilon)$ are disjoint. Therefore, $\xi$ is injective.

Lemma \ref{lem:ThanksFanny} also gives us that
\[
W = \bigcap_{1 \le i \le 2 r} B((\gamma_i^*)^-, \epsilon)\subset\RPddual
\]
is nonempty. The same argument as above applied to the dual representation gives $\theta$.

Given two infinite reduced words $x=(x_i)$ and $y=(y_i)$, let $r = r(x, y)$ be the first index at which they differ. Then
\[
\xi(x) \in x_0 \cdots x_{r - 1}( b(x_r^+, \epsilon)) \quad\textrm{and}\quad\theta(y) \in x_0 \cdots x_{r - 1} ( b((y_r^*)^+, \epsilon)).
\]
If $X=(x_0x_1\ldots x_r)^{-1}(\xi(x))$ and $Y=\left( (x_0x_1\cdots x_{r_1})^{-1}\right)^*(\theta(y))$, then Lemma \ref{lem:ThanksFanny} implies that $X$ and $Y$ span $\R^d$. Therefore, since \hbox{$x_0 \cdots x_{r - 1}$} is a linear transformation, $\xi(x)$ and $\theta(y)$ span $\R^d$. This implies that $\xi$ and $\theta$ are transverse.
\end{proof}

If $\Gamma$ is irreducible, it follows immediately from Proposition \ref{convex irreducible} that the injection of $\Gamma$ into $\SL_d(\R)$ guaranteed by Proposition \ref{prop:LimitMaps} is projective Anosov. However, this is not necessarily the case. We also need the following estimate of Benoist. We note that in Benoist \cite{BenoistAnnals} the Cartan and Jordan projections map into $\exp(\bar{\mathfrak{a}}^+)$.

\begin{Proposition}[Benoist \cite{BenoistAnnals}]
\label{prop:BenoistOperatorNorm}
If $\epsilon>0$ and $S$ is $\epsilon$-Schottky, then there is a constant $C_\epsilon \in (0, 1)$ such that if $x_0 \cdots x_n$
is a reduced word in $S$, then $x_0 \cdots x_n$ is proximal and
\[
\frac{||x_0 \cdots x_n||}{||x_0|| \cdots ||x_n||} \ge C_\epsilon^{n+2}.
\]
\end{Proposition}

\begin{proof}
Proposition 6.4 in Benoist \cite{BenoistAnnals} asserts that 
 $x_0 \cdots x_n$ is proximal and
there exists $C_\epsilon>0$ so that
\[
\frac{||x_0 \cdots x_n||}{\Lambda_1(x_0)\cdots \Lambda_1(x_n)} \ge C_\epsilon^{n+2}.
\]
Since $1<\Lambda_1(x_i)\le ||x_i||$ for all $i$ (see, for example, \cite[Cor.\ 6.3]{BenoistAnnals}), the estimate follows.
\end{proof}

We now establish our criterion for a collection of elements to generate a projective Anosov Schottky group.

\begin{Proposition}\label{convexanosov}
If $\epsilon>0$ and $S = \{\gamma_1, \dots, \gamma_{2 r}\}$ and $S^* = \{\gamma_1^*, \dots, \gamma_{2 r}^*\}$
are $\epsilon$-Schottky symmetric sets of proximal matrices
in $\SL_d(\R)$ and $\SL((\R^d)^*)$ such that
\[
\alpha_1(\mu(\gamma_i)) \ge -3 \log(C_\epsilon)
\]
for all $i$, where $C_\epsilon$ is the constant in Proposition \ref{prop:BenoistOperatorNorm},
then the group $\Gamma$ generated by $S$ is free of rank $r$ and the injection of $\Gamma$ into $\SL_d(\R)$ is projective Anosov.
\end{Proposition}

\begin{proof}
Proposition \ref{prop:LimitMaps} guarantees that $\Gamma$ is a free group of rank $r$ and that there exist 
transverse, dynamics preserving, $\Gamma$-equivariant, injective continuous maps
\[
\xi : \partial_\infty \Gamma \to \RPd \quad \textrm{and}\quad\theta: \partial_\infty \Gamma \to \RPddual.
\]
Theorem \ref{GGKW criterion} then implies that if
\[
\lim \alpha_1(\mu(x_0\ldots x_n)) = +\infty
\]
for every infinite reduced word $(x_n)$ in $S$, then the injection of $\Gamma$ into $\SL_d(\R)$ is projective Anosov.

Let $\bigwedge^2 g$ be the second exterior product of $g \in \SL_d(\R)$. Then
\[
\alpha_1(\mu(g)) = \log \frac{||g||^2}{||\bigwedge^2 g||},
\]
since $\log ||g||=\mu_1(g)$ and
\[
\log ||\tiny{\bigwedge}{}^2g||=\log||\tiny{\bigwedge}{}^2(\exp(\mu(g)))||=\mu_1(g)+\mu_2(g).
\]
Therefore,
\begin{align*}
\alpha_1(\mu(x_0 \cdots x_n)) &= \log \frac{\|x_0 \cdots x_n\|^2}{\|\bigwedge^2(x_0 \cdots x_n)\|} \\
&= \log \frac{\|x_0 \cdots x_n\|^2}{\|x_0\|^2 \cdots \|x_n\|^2} + \log \frac{\|x_0\|^2 \cdots \|x_n\|^2}{\|\bigwedge^2(x_0 \cdots x_n)\|} \\
&\ge 2 (n+2) \log C_\epsilon + \log \frac{\|x_0\|^2 \cdots \|x_n\|^2}{\|\bigwedge^2 x_0\| \cdots \|\bigwedge^2 x_n\|} \\
&= 2 (n+2) \log C_\epsilon + \sum_{i = 0}^n \alpha_1(\mu(x_i)) \\
&\ge (2n+4 - 3(n+1)) \log C_\epsilon = -(n-1) \log C_\epsilon.
\end{align*}
Since $C_\epsilon \in (0, 1)$, we have that $-(n-1) \log C_\epsilon \to \infty$ as $n \to \infty$, so
\[
\lim_{n \to \infty} \alpha_1(\mu(x_0\cdots x_n))=+\infty.
\]
This completes the proof.
\end{proof}

\medskip
\noindent
{\bf Remark:} Our original proof of Proposition \ref{convexanosov} showed that, in our setting, one could use the fact \hbox{$\lim \alpha_1(\mu(x_0\ldots x_n)) = +\infty$} for every infinite reduced word $(x_n)$ in $S$ to verify condition (iv) of Proposition 3.16 in Guichard--Wienhard \cite{Guichard-Wienhard}. As the results of Gu\'eritaud, Guichard, Kassel and Wienhard \cite{GGKW} are much more general, we make use of their work instead.

\medskip

Theorem \ref{thm:ConvexSchottky} then follows immediately from Proposition \ref{convexanosov} and the following lemma.

\begin{Lemma}\label{lem:SchottkyPower}
Let $S = \{\gamma_1, \dots, \gamma_{2 r}\}$ be a well-positioned symmetric set of biproximal matrices in $\SL_d(\R)$. Then there exists $\epsilon \in (0, 1)$ and $N > 0$ such that if $n\ge N$, then
\begin{enumerate}

\item
$S^n$ and $(S^n)^*$ are $\epsilon$-Schottky, and

\item
$\alpha_1(\mu(\gamma_i^n)) \ge -3 \log(C_\epsilon)$ for all $i$, where $C_\epsilon$ is the constant in Proposition \ref{prop:BenoistOperatorNorm}.

\end{enumerate}
\end{Lemma}

\begin{proof}
Choose
\[
\epsilon_1 = \min_{\gamma_j \neq \gamma_i^{-1}} d(\gamma_i^+, \gamma_j^-)>0,
\]
and set $\epsilon = \frac{1}{6} \epsilon_1$. Notice that
if $\gamma_j \neq \gamma_i^{-1}$, then 
\[
d((\gamma_i^*)^+,(\gamma_j^*)^-))=d((\gamma_i^{-1})^-,(\gamma_j^{-1})^+)\ge \epsilon.
\]

As in the remark in Section 6.2 of \cite{BenoistAnnals}, there exists $N$ so that if $n>N$, then 
$\gamma_i^n$ and $(\gamma_i^n)^*$ are $\epsilon$-proximal for all $i$,
so $S$ and $S^*$ are $\epsilon$-Schottky symmetric sets of proximal matrices in $\SL_d(\R)$ and $\SL((\R^d)^*)$.

Benoist \cite[Cor. 2.5.2]{Benoist} showed that 
\[
\lambda(\gamma_i) = \lim_{n \to \infty} \frac{1}{n} \mu(\gamma_i^n).
\]
Since each $\gamma_i$ is proximal, $\exp(\lambda(\gamma_i))$ is proximal, so $\alpha_1(\lambda(\gamma_i)) > 0$ for all $i$. In particular, $\lim_{n\to\infty}\alpha_1(\mu(\gamma_i^n)) = +\infty$, and hence we can also choose $N$ large enough that
\[
\alpha_1(\mu(\gamma_i^n)) \ge -3 \log(C_\epsilon),
\]
for all $i$ and $n\ge N$. This completes the proof of the lemma.
\end{proof}

\medskip\noindent
{\bf Remark:}
One may use the Pl\"ucker embedding (see Proposition \ref{Plucker}) to produce analogues of Theorem \ref{thm:ConvexSchottky} for any semisimple Lie group with finite center and any pair $P^\pm$ of proper opposite parabolic subgroups. 



\bibliographystyle{plain}
\bibliography{anosovbibtex}

\end{document}